\documentclass[a4paper,10pt,oneside,notitlepage]{amsart}

\usepackage[utf8]{inputenc}
\usepackage{color}
\usepackage{amsmath} 
\usepackage{amssymb} 
\usepackage{amsthm}
\usepackage{geometry}
\usepackage{graphicx}
\usepackage{stmaryrd}
\graphicspath{{T4-figures/}}
\usepackage{esint}
\usepackage[colorlinks=true,linkcolor=blue]{hyperref}

\theoremstyle{plain}
\newtheorem{thm}{Theorem}
\newtheorem{prop}{Proposition}[section]
\newtheorem{lem}[prop]{Lemma}
\newtheorem{cor}[prop]{Corollary}

\newtheorem{rmk}[prop]{Remark}

\newcommand {\R} {\mathbb{R}} \newcommand {\Z} {\mathbb{Z}}
\newcommand {\T} {\mathbb{T}} \newcommand {\N} {\mathbb{N}}
 
\newcommand {\p} {\partial}

\DeclareMathOperator {\dist} {dist}

\DeclareMathOperator{\F} {\mathcal{F}}

\title{On the Energy Scaling Behaviour of a Singularly Perturbed Tartar Square}
\author{Angkana Rüland}
\address{Institut f\"ur Angewandte Mathematik, Ruprecht-Karls-Universit\"at Heidelberg, Im Neuenheimer Feld 205, 69120 Heidelberg, Germany}
\email{Angkana.Rueland@uni-heidelberg.de}

\author{Antonio Tribuzio}
\address{Institut f\"ur Angewandte Mathematik, Ruprecht-Karls-Universit\"at Heidelberg, Im Neuenheimer Feld 205, 69120 Heidelberg, Germany}
\email{Antonio.Tribuzio@uni-heidelberg.de}

\begin{document}
\begin{abstract}
In this article we derive an (almost) optimal scaling law for a singular perturbation problem associated with the Tartar square. As in \cite{W97, C99}, our upper bound quantifies the well-known construction which is used in the literature to prove flexibility of the Tartar square in the sense of flexibility of approximate solutions to the differential inclusion. The main novelty of our article is the derivation of an (up to logarithmic powers matching) ansatz free lower bound which relies on a bootstrap argument in Fourier space and is related to a quantification of the interaction of a nonlinearity and a negative Sobolev space in the form of ``a chain rule in a negative Sobolev space''. Both the lower and the upper bound arguments give evidence of the involved ``infinite order of lamination''.
\end{abstract}
\maketitle

\section{Introduction}
In this article we study a singularly perturbed variational problem for a differential inclusion associated with the Tartar square. The Tartar square, $T_4$, and more generally its siblings the $T_N$-structures, are well-known sets in matrix space with important ramifications in the calculus of variations and the theoretical study of differential inclusions \cite{S93, MS, Msyl, MS03, CK00, K1, FS08}, the theory of partial differential equations, in particular as building blocks for convex integration schemes, ranging from elliptic and parabolic equations \cite{MS1,MRV05,S07} to equations of fluid dynamics \cite{DLS09,CFG11,S11}, and with various consequences for  applications, for instance for the analysis of certain phase transformations \cite{CS13,CS15} and related differential inclusions \cite{P10, MP98}. For further applications and implications we refer to the lecture notes and survey articles \cite{M1, Ri18, KMS03, K1}.

\subsection{The Tartar square and the ``stress-free'' setting}
Let us recall the ``stress-free" set-up and some properties of our problem.
The Tartar square -- which was introduced in several places in the literature \cite{Sch75, AH86, NM91, CT93, T93} (see also the survey articles from above) -- is the following set $\mathcal{K}\subset \R^{2\times 2}$:
\begin{equation}\label{sf-states}
\mathcal{K}:=\Big\{A_1,A_2,A_3,A_4\Big\}
\quad \text{with} \quad
A_1=\begin{pmatrix} 
-1 & 0\\
0 & -3
\end{pmatrix},\,
A_2=\begin{pmatrix} 
-3 & 0\\
0 & 1
\end{pmatrix},\,
A_3=-A_1,\,
A_4=-A_2.
\end{equation}
It displays a striking \emph{dichotomy between rigidity and flexibility} for the associated differential inclusion. On the one hand, it is easily shown (for convenience, a proof is recalled in Section \ref{sec:aux}) that any solution $u\in W^{1,\infty}(\Omega)$ to the differential inclusion
\begin{align}
\label{eq:diff_incl}
\nabla u \in \mathcal{K} \mbox{ a.e. in } \Omega
\end{align} 
is \emph{rigid} in the sense that any solution to \eqref{eq:diff_incl} is an affine function whose gradient is equal to one of the four matrices $A_1,\dots,A_4$. On the other hand, the Tartar square is \emph{flexible} on the level of approximate solutions: Indeed, it is possible to find sequences $(u_j)_{j\in \N}\in W^{1,\infty}(\Omega)$ such that $\dist(\nabla u_j, \mathcal{K}) \rightarrow 0$ in measure and such that no subsequence  of $(\nabla u_j)_{j\in \N}$ converges in measure (to a constant gradient in $\mathcal{K}$; see Section \ref{sec:upper} as well as \cite[Chapter 2.5]{M1} and \cite{W97,C99} for qualitative and quantitative versions of this construction). Moreover, it is known that arbitrarily small perturbations of the Tartar square enjoy even stronger flexibility in the sense that if $\mathcal{K}_{\delta}\subset \R^{2\times 2}$ is an arbitrarily small, open neighbourhood of $\mathcal{K}$ in $\R^{2\times 2}$, then there are infinitely many, non-affine solutions to the differential inclusion $\nabla u \in \mathcal{K}_{\delta}$ which can be obtained by the method of convex integration \cite{MS}.
These rigidity and flexibility aspects are mirrored in the algebraic properties of the set $\mathcal{K}$: On the one hand, the set $\mathcal{K}$ does not have any rank-one connections, i.e. for any $i,j \in \{1,\dots,4\}$ with $i \neq j$ it holds that $rk(A_i-A_j)=2>1$. This excludes ``trivial solutions'' to \eqref{eq:diff_incl} such as simple laminates. It furthermore directly implies that the lamination convex hull $\mathcal{K}^{lc}$ of $\mathcal{K}$ is trivial. On the other hand, however, the rank-one convex hull is non-trivial: 
\begin{align*}
\begin{split}
\mathcal{K}^{rc} &= \text{conv}(\{ P_1, P_2, P_3, P_4\})\cup \text{conv}(\{A_1,P_1\})\cup \text{conv}(\{A_2,P_2\})\\
& \quad \cup \text{conv}(\{A_3,P_3\})\cup \text{conv}(\{A_4,P_4\}), 
\end{split}
\end{align*}
where $\text{conv}(\cdot)$ denotes the convex hull and
\begin{align}
\label{eq:P}
\begin{split}
P_1 = \begin{pmatrix} -1 & 0 \\ 0 & -1 \end{pmatrix}, P_2 = \begin{pmatrix} -1 & 0 \\ 0 & 1 \end{pmatrix}, P_3 = \begin{pmatrix} 1 & 0 \\ 0 & 1 \end{pmatrix}, P_4 = \begin{pmatrix} 1 & 0 \\ 0 & -1 \end{pmatrix}.
\end{split}
\end{align}
The set $\mathcal{K}^{rc} = \mathcal{K}^{qc}$ is obtained by laminates of infinite order. 
Thus, the outlined properties make the Tartar square a prototypical model problem for studying more detailed properties of the dichotomy between rigidity and flexibility.

\subsection{The singularly perturbed problem and a scaling law}
Motivated by the long-term goal of understanding the described dichotomy and related dichotomies in the study of shape-memory alloys more precisely and quantitatively \cite{D, MS, K, CDK, DM1, DM2,DMP10, R16,RZZ16, RZZ18, DPR20}, and inspired by the observation in \cite{RTZ19} that the scaling behaviour of associated singularly perturbed problems give certain upper bounds on possible regularities of wild convex integration solutions, we here study the minimal energy scaling of a singularly perturbed Tartar square. Let us emphasize that in this context upper bound constructions are well-known and had earlier been quantified in \cite{W97} and also in \cite{C99}. We repeat these estimates in Section \ref{sec:upper} for completeness. The main novelty of our work consists in proving (essentially) matching lower scaling bounds.

Let us outline the setting of this. We begin by noting that the differential inclusion \eqref{eq:diff_incl} for the Tartar square can be rewritten in terms of characteristic functions indicating the ``phase'' the gradient is in
\begin{align*}
\nabla u = \begin{pmatrix} 
-\chi_1 +\chi_3 - 3 \chi_2 + 3 \chi_4 & 0\\
0 & - 3\chi_1 + 3 \chi_3 + \chi_2 - \chi_4
\end{pmatrix},
\end{align*}
where
\begin{equation}\label{char-functions}
\chi_j\in\{0,1\} \text{ for } j=1,\dots,4
\quad \text{and} \quad
\chi_1 + \chi_2 + \chi_3 + \chi_4 = 1.
\end{equation}
Using this formulation and motivated by Hooke's law, we consider the following \emph{elastic energy}
\begin{align}
\label{eq:e_el}
E_{el}(u,\chi):= \int\limits_{[0,1]^2}\left|\nabla u - \begin{pmatrix} 
-\chi_1 +\chi_3 - 3 \chi_2 + 3 \chi_4 & 0\\
0 &-3\chi_1 +3 \chi_3 + \chi_2 - \chi_4
\end{pmatrix}\right|^2 dx.
\end{align}
Here $u:[0,1]^2 \rightarrow \R^2$ is the ``deformation'' and  the functions $\chi_j$ are subject to the constraints from \eqref{char-functions}. Moreover, here and in the following we have used the abbreviation 
\begin{align}
\label{eq:chi_full}
\chi= \text{diag}(\chi_{1,1},\chi_{2,2}),
\quad \text{where} \quad
\begin{cases}
\chi_{1,1}=-\chi_1 +\chi_3 - 3 \chi_2 + 3 \chi_4, \\
\chi_{2,2}=-3\chi_1 +3 \chi_3 + \chi_2 - \chi_4.
\end{cases}
\end{align}
The elastic energy thus measures the deviation of a given deformation from being a solution to the differential inclusion \eqref{eq:diff_incl}. We emphasize that due to the flexibility of approximate solutions, the vanishing of the elastic energy along some sequence $(u_j)_{j\in \N}$ does however not entail that along a subsequence the gradients $(\nabla u_j)_{j\in \N}$ converge in measure against a constant map (in $\mathcal{K}$).

Heading towards a scaling result for the Tartar square, for $F\in \mathcal{K}^{qc}$ arbitrary but fixed, we  set
\begin{equation}\label{eq:e_el-chi}
E_{el}^{\boxempty}(\chi,F) := \inf\limits_{u \in \mathcal{A}^{\boxempty}_F} E_{el}(u,\chi).
\end{equation}%
Associated with this definition we consider two natural choices for the possible classes of deformations among which we minimize: Fixing the mean value of $\nabla u$ we consider
\begin{align}
\label{eq:Aper}
\mathcal{A}^{\text{per}}_F:= \{u\in W^{1,1}_{loc}(\R^2;\R^2): \nabla u \ \T^2 \text{-periodic}, \ \overline{\nabla u} = F\}, \ \overline{\nabla u}:= \int\limits_{\T^2} \nabla u(x) dx.
\end{align}
In this case we always (for instance, for the elastic energy \eqref{eq:e_el}) identify $[0,1]^2$ with the torus $\T^2$ of side length one and in addition also assume that the phase indicators $\chi_j$ are one-periodic functions.
As an alternative, we fix affine boundary conditions for $u$ and study for $F\in \mathcal{K}^{qc}$ and $b\in \R^2$
\begin{align}
\label{eq:Aaff}
\mathcal{A}^{\text{aff}}_F:= \{u\in W^{1,1}_{loc}(\R^2;\R^2): \ u(x) = Fx + b \text{ on }  \R^2 \setminus [0,1]^2\}.
\end{align}

In order to regain some rigidity, we add a singular perturbation. More precisely modelling the ``surface energy'' by
\begin{align}
\label{eq:e_surf}
E_{surf}(\chi) :=  \sum\limits_{j=1}^4 \|\nabla \chi_j\|_{TV([0,1]^2)},
\end{align}
for every sufficiently small parameter $\epsilon>0$, we consider the total energy
\begin{align}
\label{eq:e}
E_{\epsilon}^{\boxempty}(\chi,F) := E_{el}^{\boxempty}(\chi,F) + \epsilon E_{surf}(\chi).
\end{align}
We emphasize that with the convention introduced above, in the case that we consider the minimization problem in the class \eqref{eq:Aper}, also in the definition of the surface energy the set $[0,1]^2$ is identified with $\T^2$.
The surface energy, being a higher order term, regularizes the problem by penalizing fine oscillations of the phase indicators and hence provides some compactness in the problem (for fixed $\epsilon>0$).
Seeking to study quantitatively ``how rigid'' or ``how flexible'' the Tartar square is, we are interested deriving a scaling law for the minimal (total) energy as $\epsilon \rightarrow 0$. As our main result, we obtain the following (up to exponents of logarithms) matching upper and lower scaling bounds:

\begin{thm}
\label{thm:main}
Let $\chi_j$ for $j=1,\dots,4$ be as in \eqref{char-functions} and let $E_\epsilon^{\boxempty}$ be defined as in \eqref{eq:e}.
For every $\epsilon>0$ and $\nu\in(0,1)$ let
\begin{equation}\label{eq:scalings}
r_\nu(\epsilon)=\exp\big(-c|\log(\epsilon)|^{\frac{1}{2}+\nu}\big),
\quad
r(\epsilon)=\exp\big(-C|\log(\epsilon)|^\frac{1}{2}\big),
\end{equation}
where $c,C>0$ are universal constants.  Let $\mathcal{X}^{\boxempty}$ be either given by
\begin{align*}
\mathcal{X}^{\text{per}}:=\{\chi: \T^2 \rightarrow \{0,1\}: \ \chi \text{ as in \eqref{char-functions} }\} \mbox{ or } \mathcal{X}^{\text{aff}}:=\{\chi: [0,1]^2 \rightarrow \{0,1\}: \ \chi \text{ as in \eqref{char-functions} }\} .
\end{align*}
Assume that $F \in \mathcal{K}^{qc}\setminus\mathcal{K}$ and define
\begin{align*}
E_{\epsilon}^{\boxempty}(F):=\inf\limits_{\chi \in \mathcal{X}^{\boxempty}}\{E_\epsilon^{\boxempty}(\chi,F) \} .
\end{align*}
Then, for every $\nu\in(0,1)$ there holds
$$
C_1 r_\nu(\epsilon) \leq  E_{\epsilon}^{\boxempty}\leq C_2 r(\epsilon)
$$
for every $\epsilon>0$ small enough. Here the constants $0<C_1\leq C_2$ are independent of $\epsilon$ but may depend on the choice of $F$.
\end{thm}

Let us comment on this result: In contrast to other phase transition problems, our scaling law is not polynomial in the small parameter $\epsilon>0$ but of an order which is converging more slowly as $\epsilon \rightarrow 0$ than any polynomial in $\epsilon>0$. This is due to the fact that we are dealing with infinite order laminates: While any finite order laminate is expected to have a polynomial in $\epsilon$ scaling, our problem becomes degenerate in that an infinite order laminate requires strictly more oscillation that any finite order laminate. In this sense, Theorem \ref{thm:main} captures and quantifies the infinite order of lamination in our problem and thus distinguishes it from many other scaling laws in the literature on phase transformations. 

The infinite order of lamination is also directly reflected in our proof of Theorem \ref{thm:main}. On the one hand, it directly enters in the (well-known, here quantified) upper bound construction. However, it also enters in a more subtle way in the main novelty of our article, the lower bound for which we use a bootstrap iteration argument. Although the lower bound necessitates an ansatz-free argument, it is still strongly reminiscent of the upper bound construction and also the staircase laminate argument from \cite{CFM05}. Seeking to mimic the rigidity argument for the ``stress-free'' differential inclusion \eqref{eq:diff_incl} in which one uses that the $\p_1 u_1$ and the $\p_2 u_2$ components of any solution to \eqref{eq:diff_incl} ``determine'' each other, we are lead to an interesting Fourier space ``chain rule problem'' in negative-order Sobolev spaces. Careful quantitative bootstrap type estimates for this problem then provide the central argument for our lower bound.

Our result (and model) thus serves as an extreme case compared to other scaling laws for differential inclusions in phase transformations in that it is an ``extremely expensive'' construction.

\subsection{Relation to the literature}
Our result should be viewed in the context of scaling laws in the calculus of variations in general and more specifically in the modelling of shape-memory alloys and related phase transformation problems (see \cite{K07} and \cite{M1} for surveys on this). In the context of the modelling of shape-memory alloys, scaling laws, providing some insights on the possible behaviour of energy minimizers, have been deduced in various settings \cite{KM1, KM2,CO, CO1, CC14,CC15, CZ16, CDZ17, CDMZ20, BG15, KW14, KW16, KO19, KK, KKO13, L01, Rue16b}. For certain models, in subsequent steps, even finer properties (such as for instance almost periodicity results) have been derived \cite{C1}.
While these methods have provided important insight into many physically relevant problems, none of the known scaling bounds deal with problems in which a dichotomy between rigidity and flexibility is known. Our problem thus addresses a weak form of this dichotomy for the first time in a model case.
Moreover, we emphasize that while our result does not directly model a martensitic phase transformation, it is strongly motivated by the commonly used differential inclusions and the arising mictrostructures as, for instance, used in describing these problems in the stress-free setting \cite{B, B3}. We emphasize that, for instance, in the (geometrically linearized) cubic-to-monoclinic phase transformation, it was shown that closely related $T_3$-structures appear \cite{CS13, CS15}, for whose more quantitative analysis our investigation seems to be a natural preliminary step.

\subsection{Outline of the article}
The remainder of the article is structured as follows:  In Section \ref{sec:upper} we first provide a quantitative version of the (well-known) upper bound construction for the flexibility of the Tartar square. In Section \ref{sec:lower}, after briefly recalling the (well-known) rigidity argument for the Tartar square and auxiliary properties of the associated elastic energy in Section \ref{sec:aux}, as the main novelty of our article, we complement the scaling of the upper bound construction with a (nearly) matching lower bound.

\section{An Upper Bound Construction}
\label{sec:upper}

In this Section we quantify the total energy
of the well-known construction of infinite orders of laminations which is used in the literature to prove flexibility of the Tartar square (see, for instance, \cite[Section 2.5]{M1}). We stress that this quantitative construction had first been quantified in the literature in \cite{W97, C99} (for closely related continuum and finite element models) and that we recall it for completeness here.
This construction is an example of a sequence with vanishing elastic energy which is not strongly compact.
Balancing the elastic and the surface energy terms through a parameter optimization, as in \cite{W97,C99}, we obtain an upper bound (in terms of scaling) of our perturbed problem both for the affine and the periodic settings.

\subsection{Quantification of the total energy of the infinite-order laminate}
\label{sec:ub-infty}
In what follows we take into account zero boundary conditions, that is we will define $u_\epsilon\in\mathcal{A}^{\rm aff}_0$ and $\chi_\epsilon\in\mathcal{X}^{\rm aff}$ for every $\epsilon>0$ and quantify
\begin{equation}\label{eq:ub_tot-e}
E_\epsilon(u_\epsilon,\chi_\epsilon):=E_{el}(u_\epsilon,\chi_\epsilon)+\epsilon E_{surf}(\chi_\epsilon).
\end{equation}
The argument is completely analogous for any other affine boundary datum $Fx +b$ with $F\in\mathcal{K}^{qc}\setminus\mathcal{K}$, $b\in \R^2$.
Further the construction directly provides the upper-bound estimate of Theorem \ref{thm:main} for both $E^{\rm aff}_\epsilon$ and $E^{\rm per}_\epsilon$ by taking the $\T^2$-periodic extension of $u_\epsilon$ and $\chi_\epsilon$.

Since $\mathcal{K}$ does not have rank-one connections we make use of some auxiliary matrices (in particular the matrices from \eqref{eq:P}) to build laminates of higher and higher order, reducing the volume fraction of the region in which the gradients differ from elements of $\mathcal{K}$ but increasing the surface energy.

\smallskip

\begin{figure}[t]
\includegraphics{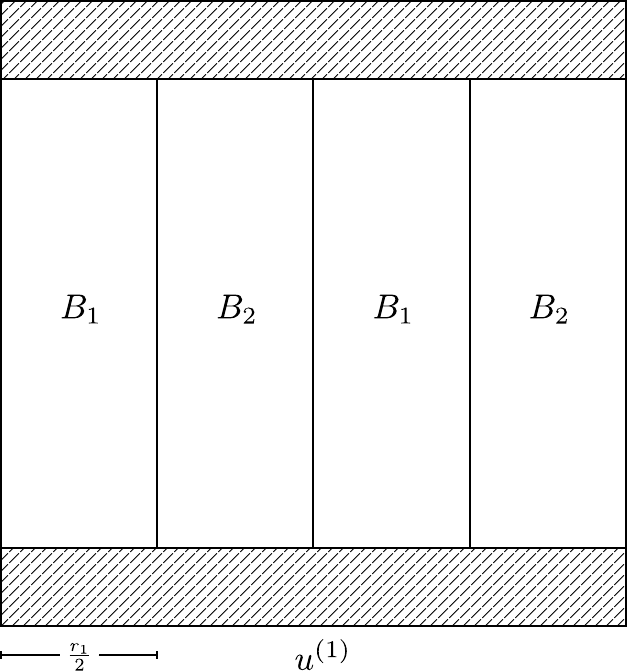}
\qquad
\includegraphics{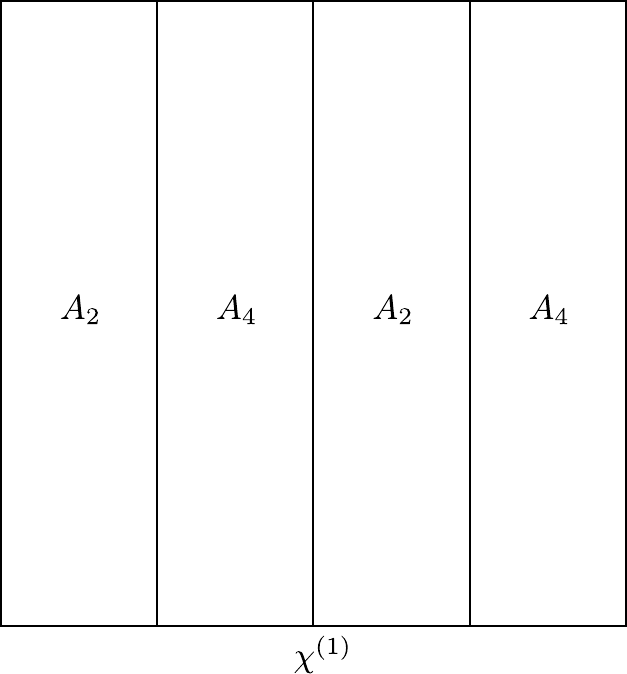}
\caption{On the left the first-order laminate construction $u^{(1)}$.
The shaded regions represent the cut-off areas.
On the right the projection of $\nabla u^{(1)}$ onto $\mathcal{K}$.}
\label{fig:ub-first}
\end{figure}

\emph{First-order laminate.}
Let $0<r_1<\frac{1}{2}$ be an arbitrarily small parameter to be determined such that $\frac{1}{r_1}$ is integer.
We resolve the boundary datum as a laminate (and a cut-off layer) with gradients
$$
B_1=\begin{pmatrix}-1&0\\0&0\end{pmatrix} \text{ and } B_2=-B_1.
$$
Any other rank-1-convex combination of elements of $\mathcal{K}^{qc}$ would lead to an analogous construction.
Thanks to the rank-1-connection between $B_2$ and $B_1$, we define the continuous function $v^{(1)}$ such that
$$
v^{(1)}(0,x_2)=v^{(1)}(r_1,x_2)=0,
\quad
\nabla v^{(1)}(x)=\begin{cases}
B_1 & x\in[0,\frac{r_1}{2}]\times[0,1], \\
B_2 & x\in[\frac{r_1}{2},r_1]\times[0,1],
\end{cases}
$$
and consider, without relabeling, its $r_1$-periodic (in the $x_1$ variable) extension on $[0,1]^2$.
We then use a cut-off argument to attain zero boundary conditions on the whole $\p[0,1]^2$ by setting $u^{(1)}\in W^{1,\infty}_0([0,1]^2;\R^2)$ as
$$
u^{(1)}(x_1,x_2)=\Big(\varphi\Big(\frac{x_2}{r_1}\Big)v^{(1)}(x_1,x_2)\Big)\varphi\Big(\frac{1-x_2}{r_1}\Big),
$$
where $\varphi(t)=\max(0,\min(2t,1))$.
We also set $\chi^{(1)}\in BV([0,1]^2;\mathcal{K})$ as the pointwise projection of $\nabla u^{(1)}$ on $\mathcal{K}$, see Figure \ref{fig:ub-first}.

It is convenient to view the elastic energy as the sum of two different terms.
One corresponds to the volume-fraction of the auxiliary states $B_1$ and $B_2$ and it is proportional to the area of $\{x\in[0,1]^2\,:\,\nabla u^{(1)}(x)\not\in\mathcal{K}\}$.
The other contribution is given by the cut-off and it is proportional to the area of $[0,1]\times[0,\frac{r_1}{2}]$.
Hence
$$
E_{el}(u^{(1)},\chi^{(1)}) \sim 1+r_1.
$$
The surface energy is the sum of the perimeters of $[(k-1)\frac{r_1}{2},k\frac{r_1}{2}]\times[0,1]$ for $k=1,\dots,\frac{2}{r_1}$, that is
$$
E_{surf}(\chi^{(1)}) \sim \frac{1}{r_1}.
$$

Here and in the sequel when writing $a\sim b$ we mean that $c^{-1}b\le a\le c b$ where $c$ is a fixed constant.

\smallskip

\begin{figure}[t]
\includegraphics{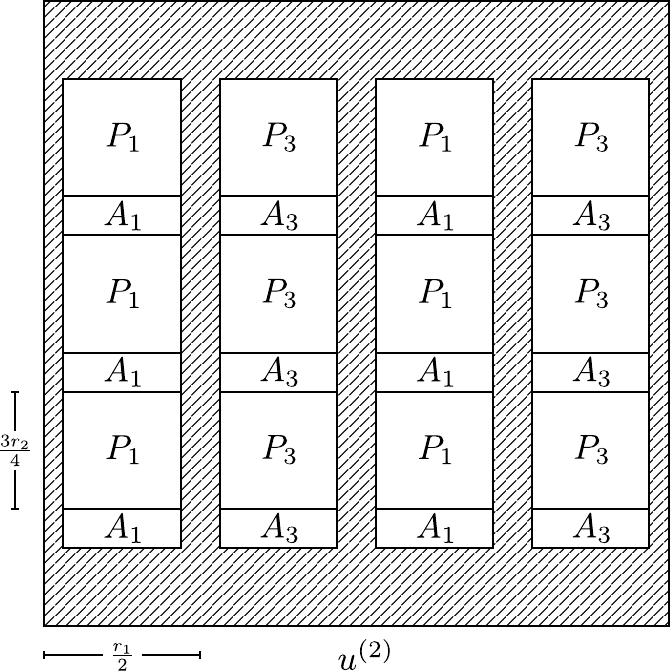}
\qquad
\includegraphics{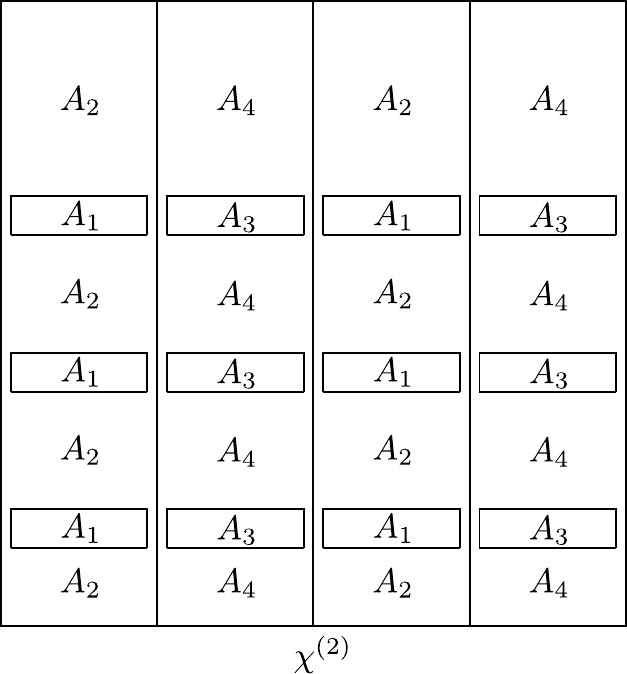}
\caption{On the left the second-order laminate construction $u^{(2)}$.
The shaded regions represent the cut-off areas.
On the right the projection of $\nabla u^{(2)}$ on $\mathcal{K}$.}
\label{fig:ub-second}
\end{figure}

\emph{Second-order laminate.}
Let $0<r_2<\frac{r_1}{2}$ be an arbitrary parameter such that $\frac{r_1}{r_2}$ is integer.
From the fact that
$$
B_1=\frac{1}{4}A_1+\frac{3}{4}P_1
\quad \text{ and } \quad
B_2=\frac{1}{4}A_3+\frac{3}{4}P_3,
$$
in each rectangle in which $\nabla u^{(1)}=B_1,B_2$ we replace $u^{(1)}$ with a simple laminate (up to cut-off) having gradients $A_1,P_1$ and $A_3,P_3$ respectively, attaining boundary conditions $u^{(1)}$.
Namely, we take $v^{(2)}$ to be continuous and such that
$$
v^{(2)}\Big(x_1,\frac{r_1}{2}\Big)=v^{(2)}\Big(x_1,r_2+\frac{r_1}{2}\Big)=B_1x,
\quad
\nabla v^{(2)}(x):=\begin{cases}
A_1 & x\in[0,\frac{r_1}{2}]\times[0,\frac{1}{4}r_2]+(0,\frac{r_1}{2}), \\
J_1 & x\in[0,\frac{r_1}{2}]\times[\frac{1}{4}r_2,r_2]+(0,\frac{r_1}{2}).
\end{cases}
$$
We consider its $r_2$-periodic (in the $x_2$ variable) extension on $[0,\frac{r_1}{2}]\times[\frac{r_1}{2},1-\frac{r_1}{2}]$ and put $v^{(2)}=u^{(1)}$ on $[0,\frac{r_1}{2}]\times\big([0,\frac{r_1}{2}]\cup[1-\frac{r_1}{2},1]\big)$.
We obtain $u^{(2)}\in W^{1,\infty}_0([0,1]^2;\R^2)$ after a cut-off argument in $\big([0,\frac{r_2}{2}]\cup[\frac{r_1}{2}-\frac{r_2}{2}]\big)\times[\frac{r_1}{2},1-\frac{r_1}{2}]$ and repeating this analogous construction in all the other parts of the rectangle in which $\nabla u^{(1)}=B_1,B_2$, using the rank-1-connection between $A_3$ and $P_3$ where $\nabla u^{(1)}=B_2$.
Set $\chi^{(2)}\in BV([0,1]^2;\mathcal{K})$ the projection of $\nabla u^{(2)}$ on $\mathcal{K}$, see Figure \ref{fig:ub-second}.

The elastic energy is given by the sum of two terms; the area of $\big\{x\in[0,1]^2\,:\,\nabla u^{(2)}(x)=P_1,P_3\big\}$ and the contribution given by the cut-off.
The energy of the cut-off of the current step gives a contributions of order $r_2$ for every rectangle in which $\nabla u^{(1)}=B_1,B_2$ that are $\frac{2}{r_1}$ many; namely,
$$
E_{el}(u^{(2)},\chi^{(2)}) \sim \frac{3}{4}+r_1+\frac{r_2}{r_1}.
$$
The surface energy is controlled by the perimeters of the rectangles in which $\nabla u^{(2)}$ is constant; indeed the rank-1-convexity of the cut-off process yields that, connecting $\nabla u^{(2)}$ to the boundary data (i.e., $B_1$ and $B_2$), the projection of $\nabla u^{(2)}$ changes at most once.
There are $\frac{4}{r_1 r_2}$ such rectangles each of perimeter of order $r_1$. 
Hence,
$$
E_{surf}(\chi^{(2)}) \sim \frac{1}{r_2}.
$$

\smallskip

\emph{m-th-order laminate.}
We define $u^{(m)}\in W^{1,\infty}_0([0,1]^2;\R^2)$ through an iterative procedure starting from $u^{(2)}$.
Thanks to the relation
$$
P_{j'}=\frac{1}{2}A_j+\frac{1}{2}P_j,
\quad
j'=\begin{cases}j+1&j=1,2,3,\\1&j=4,\end{cases}
$$
we replace $u^{(m-1)}$, in the rectangles in which $\nabla u^{(m-1)}=P_{j'}$, with $r_m$-periodic laminate of gradients $A_j$ and $P_j$ obtaining $u^{(m)}\in W^{1,\infty}_0([0,1]^2;\R^2)$ after a cut-off argument to attain $u^{(m-1)}$ at the boundary of each rectangle.
Here $r_m$ is an arbitrarily small parameter with $0<r_m<\frac{r_{m-1}}{2}$ and $\frac{r_{m-1}}{r_m}$ integer.
We then set $\chi^{(m)}\in BV([0,1]^2;\mathcal{K})$ the projection of $\nabla u^{(m)}$ on $\mathcal{K}$.

For every $m\ge3$ we have that
\begin{equation}\label{eq:ub_vol-frac}
\big|\{x\in[0,1]^2\,:\,\nabla u^{(m)}(x)=P_1,\dots,P_4\}\big|\sim\frac{1}{2}\big|\{x\in[0,1]^2\,:\,\nabla u^{(m-1)}(x)=P_1,\dots,P_4\}\big|.
\end{equation}
The volume fraction of the cut-off regions of the $m$-th step is $2\frac{r_m}{r_{m-1}}$.
Thus its contribution in the elastic energy is $2\frac{r_m}{r_{m-1}}\big|\{x\in[0,1]^2\,:\,\nabla u^{(m-1)}(x)=P_1,\dots,P_4\}\big|$.
Hence
$$
E_{el}(u^{(m)},\chi^{(m)}) \sim 2^{-m} + \sum_{j=2}^m 2^{-j}\frac{r_j}{r_{j-1}}+r_1.
$$
The surface energy is proportional to the sum of the perimeters of the rectangles in which $\nabla u^{(m)}=P_1,\dots,P_4$.
Denoting with $R_m\in\N$ the number of such rectangles, we have
$$
R_m\sim\frac{4}{r_m r_{m-1}}\big|\{x\in[0,1]^2\,:\,\nabla u^{(m)}(x)=P_1,\dots,P_4\}\big|.
$$
Since the perimeter of each rectangle is of order $r_{m-1}$ we get
$$
E_{surf}(\chi^{(m)})\sim 2^{-m}\frac{1}{r_m}.
$$
Notice that the factor $2^{-m}$ comes from the decreasing volume in \eqref{eq:ub_vol-frac} and will not affect the optimal scaling in $\epsilon$.
The total energy of the construction above is therefore
$$
E_\epsilon(u^{(m)},\chi^{(m)})\sim 2^{-m} + \Big(\sum_{j=2}^m 2^{-j}\frac{r_j}{r_{j-1}}+r_1\Big)+\epsilon 2^{-m}\frac{1}{r_m}
$$
and it depends on $\{r_j\}_{j=1}^m$.
In order to obtain a good upper bound, we determine the optimal choice of such parameters in terms of $r_1$.

Comparing the terms $r_1$ and $\frac{r_2}{r_1}$ we get $r_2\le r_1^2$.
Since the energy depends on $r_2$ in only another term, that is $\frac{r_3}{r_2}$, the choice $r_2\sim r_1^2$ is optimal.
Working inductively, we get $r_j\sim r_1^j$.
Thus, we denote with $u_{m,r}$ and $\chi_{m,r}$ the functions $u^{(m)}$ and $\chi^{(m)}$ defined as above, corresponding to $r_j=r^j$, where $r>0$ is a small parameter.
Hence, we have
\begin{equation}\label{eq:ub_e}
E_\epsilon(u_{m,r},\chi_{m,r})\sim 2^{-m}+r+\epsilon 2^{-m}r^{-m}.
\end{equation}

\subsection{Determination of the length scale}
From the analysis performed above we obtain the following result, which provides an upper (scaling) bound for Theorem \ref{thm:main}.

\begin{prop}\label{prop:upp-bound}
Let $E_\epsilon$ and $r(\epsilon)$ be defined as in \eqref{eq:ub_tot-e} and \eqref{eq:scalings} respectively.
For every $\epsilon>0$ small enough and every $F\in\mathcal{K}^{qc}\setminus\mathcal{K}$, there exist $u_\epsilon\in\mathcal{A}^\boxempty_F$ and $\chi_\epsilon\in\mathcal{X}^\boxempty$ such that
$$
E_\epsilon(u_\epsilon,\chi_\epsilon) \sim r(\epsilon).
$$
\end{prop}
\begin{proof}
As already noticed, it is sufficient to consider affine boundary conditions.
The result comes form a parameter optimization in terms of $\epsilon$ for the constructions $u_{m,r}$ and $\chi_{m,r}$ defined in Subsection \ref{sec:ub-infty}.
Determine first the optimal length scale $r$ for the $m$-th iteration by comparing the terms $r$ and $\epsilon 2^{-m}r^{-m}$ in \eqref{eq:ub_e}, obtaining $r\sim\epsilon^\frac{1}{m+1}$.
We now look for the optimal order of iterations $m_\epsilon$.
From $2^{-m}\sim\epsilon^\frac{1}{m+1}$ we get $m_\epsilon\sim|\log(\epsilon)|^\frac{1}{2}$.
This gives
$$
r_\epsilon\sim \epsilon^{c|\log(\epsilon)|^{-\frac{1}{2}}}=\exp\big(-c|\log(\epsilon)|^\frac{1}{2}\big)
$$
for some $c>0$,
which yields the result for $F=0$ by \eqref{eq:ub_e} by taking $u_\epsilon=u_{m_\epsilon,r_\epsilon}$ and $\chi_\epsilon=\chi_{m_\epsilon,r_\epsilon}$.

The construction corresponding to a non-zero boundary datum differs from $u^{(m)}$ and $\chi^{(m)}$ only in the first step, i.e. $m=1$, being then completely analogous.
Thus it does not affect the scaling of $E_\epsilon$, hence the result is proved.
\end{proof}

\begin{rmk}
We note that $r(\epsilon)$ is smaller than any logarithmic scale and greater than any power of $\epsilon$.
Indeed, given $0<\alpha\le1$ we get
$$
\lim_{\epsilon\to0^+}\frac{\exp(-c|\log(\epsilon)|^\frac{1}{2})}{\epsilon^\alpha}=\lim_{t\to+\infty} e^{\alpha t - c\sqrt{t}}=+\infty
$$
and
$$
\lim_{\epsilon\to0^+}\frac{\exp(-c|\log(\epsilon)|^\frac{1}{2})}{\frac{1}{|\log(\epsilon)|^\alpha}}=\lim_{t\to+\infty} e^{-c\sqrt{t}}t^\alpha=0.
$$
Hence,
$$
\epsilon^\alpha\ll r(\epsilon)\ll |\log(\epsilon)|^{-\alpha}, \quad \text{for every } 0<\alpha\le1.
$$
\end{rmk}

\begin{rmk}
We do not claim that our constant $c>0$ is optimal. It is expected that this depends on the finer properties of the upper bound construction, e.g. on using branched constructions instead of direct laminations. Since the value of the constant $c>0$ is not the main emphasis of our scaling result, we do not pursue this further in this article.
\end{rmk}

\section{A Qualitative Rigidity Argument and Some Auxiliary Results for the Elastic Energy}

\label{sec:aux}

In this section, we recall an argument for the exactly stress-free rigidity of the Tartar square which will serve as our guideline for the lower bound estimate. Additionally, we will recall the expression of the elastic energy in Fourier space for different affine boundary conditions which will become a central ingredient in our quantitative lower bound arguments.

\subsection{A qualitative rigidity argument}

We recall a qualitative rigidity argument which we will mimic in our lower bound estimate.

\begin{prop}\label{prop:qual-rig1}
Let $u\in W^{1,\infty}_{loc}(\R^2;\R^2)$ be a solution of the differential inclusion
$$
\nabla u\in\mathcal{K} \mbox{ a.e. in } [0,1]^2,
$$
then $\nabla u$ is a constant matrix.
In particular $u(x)=A_i x +b$ for some $i=1,\dots,4$ and $b\in \R^2$.
\end{prop}
\begin{proof}
We follow the approach used in \cite[proof of Theorem 2.5]{M1}.
From the fact that the elements of $\mathcal{K}$ are diagonal matrices we deduce
$$
\p_2 u_1 = 0, \quad \p_1 u_2 = 0,
$$
thus
$$u_1(x_1,x_2) = f_1(x_1), \quad u_2(x_1,x_2) = f_2(x_2)$$
for some $f_1,f_2:\R\to\R$.
Hence, we obtain
\begin{align}
\begin{split}
\label{eq:grad_const}
\p_1 u_1(x_1,x_2) = f_1'(x_1) = -\chi_1(x_1,x_2) +\chi_3(x_1,x_2) - 3\chi_2(x_1,x_2) + 3\chi_4(x_1,x_2),\\
\p_2 u_2(x_1,x_2) = f_2'(x_2) = -3\chi_1(x_1,x_2) +3\chi_3(x_1,x_2) + \chi_2(x_1,x_2) - \chi_4(x_1,x_2).
\end{split}
\end{align}
We note that every matrix of $\mathcal{K}$ is completely identified by any of its diagonal entries, thus $-\chi_1 +\chi_3 - 3 \chi_2 + 3 \chi_4$ changes if and only if $-3 \chi_1 +3\chi_3 + \chi_2-\chi_4 $ does.
By \eqref{eq:grad_const} this however implies that $-\chi_1 +\chi_3 - 3 \chi_2 + 3 \chi_4$ is both a function of $x_1$ only and of $x_2$ only.
Thus, it must be constant.
\end{proof}

Note that this is a particular case of the general fact that any Lipschitz solution of $\nabla u\in\mathcal{K}'$ with $\mathcal{K}'\subset\R^{n\times m}$ of cardinality $4$ whose elements are not rank-1-connected is trivial (see \cite[Theorem 7]{CK00}).

\subsection{Elastic energy in Fourier space}

We give the expression of the elastic energy $E_{el}$ defined in \eqref{eq:e_el-chi} in Fourier space with periodic boundary conditions, following a standard approach in the literature \cite{CO1,KKO13}.
It will be useful in the sequel to rewrite $E_{el}$ in terms of the diagonal entries of $\chi$, that is
\begin{equation*}
\chi_{1,1}=-\chi_1+\chi_3-3\chi_2+3\chi_4,
\quad
\chi_{2,2}=-3\chi_1+3\chi_3+\chi_2-\chi_4.
\end{equation*}

We first fix the notation: for every $\T^2$-periodic function $f:\T^2\to\R$ denote by
$$
\F(f)(k)=\frac{1}{2\pi}\int_{\T^2} f(x) e^{-2\pi \,i\, k\cdot x} dx, \quad k\in\Z^2
$$
its Fourier transform.

We will denote with $x$ the space variable and with $k$ the frequency variable.
If there is no ambiguity, we will also use the notation $\hat f=\F(f)$.

With this notation fixed, we now turn to the characterization of the elastic energy in terms of a suitable Fourier multiplier.

\begin{lem}\label{lem:e_el-old}
Let $E_{el}^{\rm per}$ be defined in \eqref{eq:e_el-chi} and $\{\chi_j\}$ be as in \eqref{char-functions} and $\T^2$-periodic.
Then for every $F\in\R^{2\times2}_{\rm sym}$ and $\chi\in\mathcal{X}^{\rm per}$ there holds
$$
E_{el}^{\rm per}(\chi,F)=\sum_{k\in\Z^2\setminus\{(0,0)\}} \frac{k_2^2}{|k|^2}|\hat\chi_{1,1}|^2+\frac{k_1^2}{|k|^2}|\hat\chi_{2,2}|^2+|\hat\chi(0)-F|^2.
$$
\end{lem}
\begin{proof}
We first notice that
\begin{align*}
E_{el}^{\rm per}(\chi,F) &= \inf\Big\{\int_{\T^2} |\nabla u-F+F-\chi|^2dx : \nabla u \ \T^2\text{-periodic}, \ \overline{\nabla u}=F\Big\} \\
&= \inf\Big\{\int_{\T^2} |\nabla v-(\chi-F)|^2dx : v \ \T^2\text{-periodic}, \ \overline{\nabla v}=0\Big\}.
\end{align*}
We can thus rewrite $E_{el}^{\rm per}(\chi,F)$ in Fourier space as follows
\begin{equation}\label{eq:e_el-fourier}
E_{el}(v,\chi)=\sum_{k\in\Z^2}|\hat v\otimes ik-\hat\chi|^2.
\end{equation}
By minimizing \eqref{eq:e_el-fourier} in $\hat u$ we obtain
$$
(\hat v\otimes ik-\hat\chi):\hat w\otimes ik=0, \quad k\in\Z^2\setminus\{(0,0)\},
$$
for every test function $w\in L^2(\T^2;\R^2)$, which reads
$$
(\hat v\otimes ik)k=\hat\chi k, \quad k\in\Z^2\setminus\{(0,0)\}.
$$
This is solved by
$$
\hat v_1=\frac{-ik_1}{|k|^2}\hat\chi_{1,1}, \quad \hat v_2=\frac{-ik_2}{|k|^2}\hat\chi_{2,2}.
$$
Substituting these values into \eqref{eq:e_el-fourier} we get the result.
\end{proof}

In view of the lower-bound estimate which is formulated in Theorem \ref{thm:lower_bound} (see the proof in Section \ref{sec:lower_proof}), it is worth noting that from the characterization given in Lemma \ref{lem:e_el-old} we have
$$
E_{el}^{\rm per}(\chi,F) \ge \|\p_2\chi_{1,1}\|^2_{\dot H^{-1}}+\|\p_1\chi_{2,2}\|^2_{\dot H^{-1}}.
$$

\section{A Bootstrap Argument and a Proof of the Lower Bound}
\label{sec:lower}

In this section, we turn to the proof of the lower bound of Theorem \ref{thm:main}. To this end, we will make use of a bootstrap argument which again highlights the infinite order of lamination in our solutions.

In the following Sections \ref{sec:chain}-\ref{sec:chain_proof} and \ref{sec:lower_proof}, we first consider the case of periodic data and prove a lower bound for $E^{\text{per}}_{\epsilon}$. Noting that $E^{\text{per}}_{\epsilon} \leq E^{\text{aff}}_{\epsilon}$ then also leads to the desired lower bound result of Theorem \ref{thm:main} in the case of affine boundary conditions (see Section \ref{sec:lower_proof} for the details).

\subsection{A chain rule argument in $H^{-1}$}
\label{sec:chain}
In this section, we prove the following main result which will be used to prove the lower bound of Theorem \ref{thm:main} in Section \ref{sec:lower_proof}.

\begin{prop}\label{prop:H-1aim}
Let $f_1, f_2\in L^\infty(\T^2)\cap BV(\T^2)$ and let $g, h:\R\to\R$ be nonlinear polynomials with $g(0)=0=h(0)$ such that $f_2=g(f_1)$ and $f_1=h(f_2)$.
For any $\alpha \in (0,1)$, if
\begin{equation}
\label{eq:H-1aim1}
\|\p_1 f_1\|^2_{\dot{H}^{-1}} + \|\p_2 f_2\|^2_{\dot{H}^{-1}} \leq \delta
\quad \text{and} \quad
\|\nabla f_1\|_{TV}+\|\nabla f_2\|_{TV}\le\beta,
\end{equation}
then there exist $c_0\in \R$ and $C_0>0$ such that
\begin{equation}
\label{eq:H-1aim2}
\|f_1-c_0\|^2_{L^2} \leq \left( \frac{4C_0 }{\alpha^{4d}}  \right)^{\frac{1}{\alpha}}\epsilon^{-2\alpha}\max\{(\delta + \epsilon\beta)^\frac{1}{2},\delta+\epsilon\beta\},
\end{equation}
for every $\epsilon>0$ small enough.
\end{prop}

\begin{rmk}
We remark that the estimate \eqref{eq:H-1aim2} in particular implies lower bounds in $\epsilon>0$ which are more slowly decreasing than any polynomial power of $\epsilon$ (if $\|f_1-c_0\|^2_{L^2} \geq \tilde{c}>0$ independently of $\epsilon>0$).
\end{rmk}

Throughout this section, all the norms are restricted to $\T^2$.
Also, for the sake of simplicity, we assume $g$ and $h$ to be polynomials of the same degree $d\in\N$, $d\ge 2$. Without loss of generality, we may further assume that $f_1 \neq 0 \neq f_2$ since else the statement follows directly.

We will use that the first inequality in \eqref{eq:H-1aim1} can be phrased in the following two equivalent formulations
$$
\|\p_1 f_1\|^2_{\dot{H}^{-1}} + \|\p_2 g(f_1)\|^2_{\dot{H}^{-1}} \leq \delta
\quad \text{and} \quad
\|\p_1 h(f_2)\|^2_{\dot{H}^{-1}} + \|\p_2 f_2\|^2_{\dot{H}^{-1}} \leq \delta.
$$
This will yield intermediate bounds on the sets where the $L^2$-mass of $\hat f_1$ and $\hat f_2$ concentrate which will lead to \eqref{eq:H-1aim2} thanks to a bootstrap argument.

Optimizing the right-hand-side of \eqref{eq:H-1aim2} in the parameter $\alpha$, as a consequence of Proposition \ref{prop:H-1aim} we obtain the desired lower-bound estimate of Theorem \ref{thm:main} for the case $f_2=\chi_{1,1}$ and $f_1=\chi_{2,2}$ (see Section \ref{sec:lower_proof}).

\begin{rmk}
In the afore mentioned case $f_1=\chi_{2,2}$ and $f_2=\chi_{1,1}$, the polynomials $g$ and $h$ for which the relations $f_2=g(f_1)$ and $f_1=h(f_2)$ hold true can be chosen via interpolation, \emph{e.g.}
$$
g(t)=\frac{5}{12}x^3-\frac{41}{12}x,
\quad
h(t)=-g(t).
$$
Such choices of $g$ and $h$ work both in the Dirichlet and in the periodic settings since from the symmetry of Tartar's square $h(0)=g(0)=0$. We emphasize that in the setting of the Tartar square the choice of the functions $g,h$ is extremely non-unique.
\end{rmk}

\subsection{Preliminary considerations}
\label{sec:lower1}
Given two parameters $\mu,\mu_2>0$, we define the following compact cones in frequency space (see Figure \ref{fig:cones1})
$$
C_{1,\mu,\mu_2}:=\{k \in \R^2: \ |k_1| \leq \mu |k|,\, |k|\le\mu_2\}
\quad \text{and} \quad
C_{2,\mu,\mu_2}:=\{k \in \R^2: \ |k_2| \leq \mu |k|,\, |k|\le\mu_2\}
$$
and let $\chi_{1,\mu,\mu_2}$, $\chi_{2,\mu,\mu_2}$ be smoothed out characteristic functions of $C_{1,\mu,\mu_2}$, $C_{2,\mu,\mu_2}$, respectively, i.e. we choose $\chi_{1,\mu,\mu_2}$ as functions depending only on $\frac{k}{|k|}$ such that they are equal to one on the cones $C_{1,\mu,\mu_2}$, $C_{2,\mu,\mu_2}$, respectively, vanish outside of a slight thickening of these and are smooth outside of the origin.
With the notation $\chi_{1,\mu,\mu_2}(D)$, $\chi_{2,\mu,\mu_2}(D)$ we will denote the corresponding Fourier multipliers; \emph{i.e.}, 
$$
\chi_{j,\mu,\mu_2}(D)f(x)=\sum_{k\in\Z^2}\chi_{j,\mu,\mu_2}(k)\hat f(k) e^{2\pi \,i\, k\cdot x}, \text{ for } j=1,2
$$
for every $f\in L^2(\T^2)$.

\begin{figure}[t]
\includegraphics{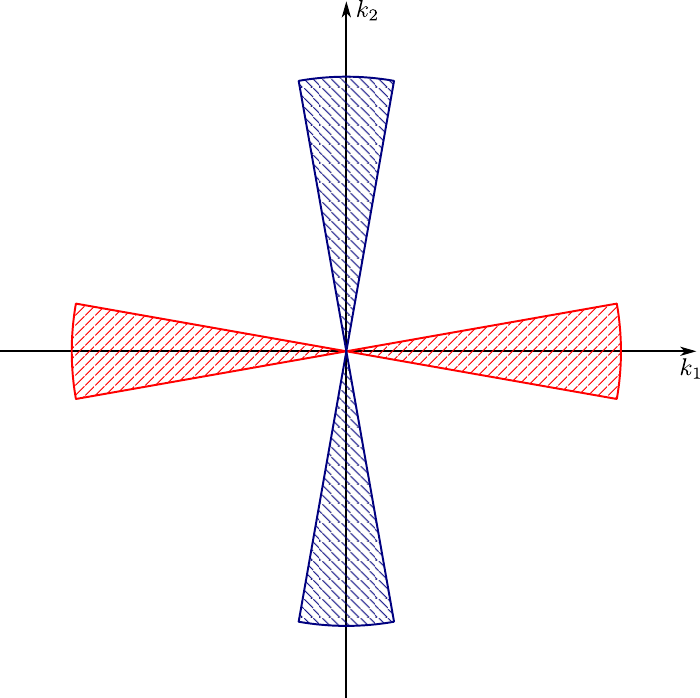}
\caption{The cones $C_{1,\mu,\mu_2}$ (blue) and $C_{2,\mu,\mu_2}$ (red), respectively.}
\label{fig:cones1}
\end{figure}

We begin our bootstrap argument, which eventually leads to the proof of Proposition \ref{prop:H-1aim}, by observing that the functions $f_1$ and $f_2$ concentrate their mass in the cones $C_{1,\mu,\mu_2}$ and $C_{2,\mu,\mu_2}$, respectively.

\begin{lem}\label{lem:imp_comp1}
Let $f_1, f_2$ and $g$ be as in the statement of Proposition \ref{prop:H-1aim}.
Then, for every $\mu, \mu_2>0$ there hold
\begin{align}
\label{eq:elastic_energy1}
\|f_1-\chi_{1,\mu,\mu_2}(D)f_1\|^2_{L^2} + \|g(f_1)-\chi_{2,\mu,\mu_2}(D)g(f_1)\|^2_{L^2} &\leq C\big(\mu^{-2} \delta+\mu_2^{-1}\beta\big) \\
\intertext{and}
\label{eq:elastic_energy2}
\|h(f_2)-\chi_{1,\mu,\mu_2}(D)h(f_2)\|^2_{L^2} + \|f_2-\chi_{2,\mu,\mu_2}(D)f_2\|^2_{L^2} &\leq C\big(\mu^{-2} \delta+\mu_2^{-1}\beta\big),
\end{align}
where $C>0$ is a constant depending on $\|f_1\|_{L^\infty}$, $\|f_2\|_{L^\infty}$, $g$ and $h$.
\end{lem}
\begin{proof}
Note first that, from $f_2=g(f_1)$ and $f_1=h(f_2)$, \eqref{eq:elastic_energy1} and \eqref{eq:elastic_energy2} are equivalent.
Therefore, it is sufficient to prove \eqref{eq:elastic_energy1}.
We divide the proof into two steps.

\emph{Step 1.} Arguing as in \cite[proof of Lemma 4.3]{KKO13}, for every $c\in\R^2$ we have that
\begin{equation*}
\begin{aligned}
\|\nabla f_1\|_{TV} &\ge \|f_1\|_{L^\infty}^{-1}\frac{1}{|c|}\int_{\T^2}|f_1-f_1(\cdot+c)|^2dx = \|f_1\|_{L^\infty}^{-1}\frac{1}{|c|} \sum_{k\in\Z^2}|(1-e^{i c\cdot k})\hat f_1(k)|^2 \\
&\ge \|f_1\|_{L^\infty}^{-1}\frac{1}{|c|} \sum_{k\in\Z^2,\, |k|>\frac{1}{L}}|(1-e^{i c\cdot k})\hat f_1(k)|^2
\end{aligned}
\end{equation*}
for every $L>0$.
Integrating over $\p B_L$ with $|c|=L$, we deduce that
$$
L^2 \|\nabla f_1\|_{TV}\ge \|f_1\|_{L^\infty}^{-1}\sum_{k\in\Z^2,\,|k|>\frac{1}{L}}|\hat f_1(k)|^2 \int_{\p B_L}|1-e^{ic\cdot k}|^2 dc \ge \|f_1\|_{L^\infty}^{-1} L \sum_{k\in\Z^2,\,|k|>\frac{1}{L}}|\hat f_1(k)|^2.
$$
Choosing $L=\mu_2^{-1}$, we infer
\begin{equation}\label{eq:lem_prel1}
\sum_{k\in\Z^2,\,|k|>\mu_2}|\hat f_1(k)|^2+\big|\mathcal{F}\big(g(f_1)\big)\big(k)|^2 \le C \mu_2^{-1}\beta.
\end{equation}

\emph{Step 2.} Passing to the frequency space we get
\begin{align*}
\|\p_1 f_1\|^2_{\dot{H}^{-1}} + \|\p_2 g(f_1)\|^2_{\dot{H}^{-1}} &= \sum_{k\in\Z^2\setminus\{(0,0)\}}\frac{k_1^2}{|k|^2}|\hat f_1(k)|^2 +\frac{k_2^2}{|k|^2}\big|\mathcal{F}\big(g(f_1)\big)(k)\big|^2 \\
& \ge \sum_{k\in\Z^2\setminus C_{1,\mu,\mu_2}}\frac{k_1^2}{|k|^2}|\hat f_1(k)|^2 + \sum_{k\in\Z^2\setminus C_{2,\mu,\mu_2}}\frac{k_2^2}{|k|^2}\big|\mathcal{F}\big(g(f_1)\big)(k)\big|^2 \\
& \ge \mu^2 \sum_{k\in\Z^2\setminus C_{1,\mu,\mu_2}}|\hat f_1(k)|^2 + \mu^2 \sum_{k\in\Z^2\setminus C_{2,\mu,\mu_2}} \big|\mathcal{F}\big(g(f_1)\big)\big(k)|^2 \\
&= \mu^2\Big(\|f-\chi_{1,\mu,\mu_2}(D)f\|^2_{L^2} + \|g(f_1)-\chi_{2,\mu,\mu_2}(D)g(f_1)\|^2_{L^2}\Big).
\end{align*}
Combining the inequality above and \eqref{eq:lem_prel1}, by \eqref{eq:H-1aim1} we obtain \eqref{eq:elastic_energy1}.
\end{proof}

Next, we seek to improve the control on the Fourier supports of $f_1$ and $f_2$ iteratively. To this end, as a crucial observation, we use that the Fourier support of $f_2$ is essentially obtained through a nonlinear function interacting with $f_1$.
If $f_1$ were such that $\hat f_1\in L^\infty(\T^2)$, this would be a consequence of the local Lipschitz continuity of $g$: Indeed, by \eqref{eq:elastic_energy1}, the fact that $f_2 = g(f_1)$ and the triangle inequality we would obtain
\begin{align}
\label{eq:compare}
\begin{split}
&\|g(\chi_{1,\mu,\mu_2}(D)f_1)- \chi_{2,\mu,\mu_2}(D)g(f_1)\|^2_{L^2}\\
&\quad \leq 2\|g(f_1)- \chi_{2,\mu,\mu_2}(D)g(f_1)\|^2_{L^2} + 2\|g(f_1)-g(\chi_{1,\mu,\mu_2}(D)f_1)\|^2_{L^2}\\
&\quad \leq 2\|g(f_1)- \chi_{2,\mu,\mu_2}(D)g(f_1)\|^2_{L^2} + 2C\|f_1-\chi_{1,\mu,\mu_2}(D)f_1\|^2_{L^2} \\
&\quad \lesssim \mu^{-2}\delta+\mu_2^{-1}\beta.
\end{split}
\end{align}
The same estimate would follow if $g$ was \emph{globally} a Lipschitz function, without requiring any further assumptions on $f_1$.

In our application we work with nonlinear functions $g$ which are only \emph{locally} Lipschitz (cubic polynomials) and we do not a priori know that $\hat f_1\in L^\infty(\T^2)$.
Hence, even though $f_1 \in L^{\infty}$, in our setting, we cannot directly proceed as in \eqref{eq:compare}, since Fourier multipliers are in general not bounded as maps from $L^{\infty}$ to $L^{\infty}$.
Yet we can still control the left-hand-side of \eqref{eq:compare} in a similar way, obtaining a (small) loss (see Corollary \ref{cor:poly_estimate}).

More precisely, in order to remedy the lack of $L^{\infty}$ bounds for $\chi_{1,\mu,\mu_2}(D)f_1$ and hence the lack of direct Lipschitz continuity arguments, we make use of Calder\'on-Zygmund estimates in $L^p$ spaces with $p\in (1,\infty)$ and interpolation. While this gives rise to a small loss, it will provide our replacement of \eqref{eq:compare} in Corollary \ref{cor:poly_estimate}.

\begin{lem}\label{lem:poly_estimate}
Let $f_1, f_2, g$ and $h$ be as in the statement of Proposition \ref{prop:H-1aim}.
Then for every $\mu,\mu'>0$ and any $\gamma \in (0,1) $ there hold
\begin{align}
\label{eq:Lip1}
\|g(f_1)-g(\chi_{1,\mu,\mu'}(D)f_1)\|_{L^2} \le \frac{C'}{\gamma^d}\|f_1-\chi_{1,\mu,\mu'}(D)f_1\|_{L^2}^{1-\gamma},
\end{align}
and
\begin{align}
\label{eq:Lip2}
\|h(f_2)-h(\chi_{2,\mu,\mu'}(D)f_2)\|_{L^2} \le \frac{C'}{\gamma^d}\|f_2-\chi_{2,\mu,\mu'}(D)f_2\|_{L^2}^{1-\gamma},
\end{align}
with $C'>0$ being a constant depending on $\|f_1\|_{L^\infty}$, $\|f_2\|_{L^\infty}$, $g$, $h$ and $d$.
\end{lem}

\begin{proof}
It is sufficient to prove the statement for $g(t)=t^d$ for some $d\in\N$, $d\ge2$.
Using the fact that $a^d-b^d=(a-b)G(a,b)$ where $G(a,b)=\sum\limits_{j=0}^{d-1} a^{d-1-j}b^{j}$ is $(d-1)$-homogeneous, by Hölder's inequality we obtain
\begin{align*}
\|g(f_1)-g(\chi_{1,\mu,\mu_2}(D)f_1)\|_{L^2} &= \|(f_1 - \chi_{1,\mu,\mu_2}(D)f_1)G(f_1,\chi_{1,\mu,\mu_2}(D)f_1) \|_{L^2} \\
&\leq \|f_1 - \chi_{1,\mu,\mu_2}(D)f_1\|_{L^{2+2\gamma}} \|G(f_1,\chi_{1,\mu,\mu_2}(D)f_1)\|_{L^{\frac{2+2\gamma}{\gamma}}}
\end{align*}
for any $\gamma\in(0,1)$ (to be fixed later).
By means of $L^p$ interpolation (see for instance \cite[Proposition 1.1.14]{Grafakos}) we get
\begin{multline*}
\|g(f_1)-g(\chi_{1,\mu,\mu_2}(D)f_1)\|_{L^2} \\
\le \|f_1-\chi_{1,\mu,\mu_2}(D)f_1\|_{L^2}^{1-\gamma} \|f_1-\chi_{1,\mu,\mu_2}(D)f_1\|_{L^{\frac{2+2\gamma}{\gamma}}}^\gamma \|G(f_1,\chi_{1,\mu,\mu_2}(D)f_1)\|_{L^{\frac{2+2\gamma}{\gamma}}}.
\end{multline*}
Invoking Hölder's inequality and the explicit form of $G(a,b)$, we further infer
\begin{align*}
\|G(f_1,\chi_{1,\mu,\mu_2}(D)f_1)\|_{L^{\frac{2+2\gamma}{\gamma}}}
&\leq \sum\limits_{j=0}^{d-1} \|f_1^{d-1-j}(\chi_{1,\mu,\mu_2}(D)f_1)^j\|_{L^{\frac{2+2\gamma}{\gamma}}}\\
&\leq \sum\limits_{j=0}^{d-1} \|f_1\|_{L^{\frac{(2+2\gamma)(d-1)}{\gamma}}}^{d-1-j} \|\chi_{1,\mu,\mu_2}(D)f_1\|_{L^{\frac{(2+2\gamma)(d-1)}{\gamma}}}^j.
\end{align*}
The $L^p$-$L^p$ boundedness of Fourier multipliers (see the Mihlin-Hörmander multiplier theorem, for instance, in \cite[Theorem 5.2.7]{Grafakos}) implies for each $j\in \{0,\dots,d-1\}$
\begin{align*}
&\|(1-\chi_{1,\mu,\mu_2}(D))f_1\|_{L^\frac{2+2\gamma}{\gamma}}^{\gamma}\|f_1\|_{L^{\frac{(2+2\gamma)(d-1)}{\gamma}}}^{d-1-j} \|\chi_{1,\mu,\mu_2}(D)f_1\|_{L^{\frac{(2+2\gamma)(d-1)}{\gamma}}}^j\\
&\le C(\gamma,d)^{j+\gamma}\|f_1\|_{L^\frac{(2+2\gamma)(d-1)}{\gamma}}^{d-1} \|f_1\|_{L^\frac{2+2\gamma}{\gamma}}^{\gamma} \leq C(\gamma,d)^{j+\gamma} \|f_1\|_{L^{\infty}}^{d-1+\gamma}.
\end{align*}
Here $1<C(\gamma,d)\leq \frac{Cd}{\gamma}$ (which follows from the Mihlin-Hörmander multiplier theorem, for instance, in \cite[Theorem 5.2.7, equation (5.2.12)]{Grafakos}) and is, in particular, independent of $\mu$ and $\mu_2$.
Combining the previous inequalities we obtain \eqref{eq:Lip1} with $C'=(Cd)^{d+1}$.
Working analogously we infer \eqref{eq:Lip2}.
\end{proof}

As a direct generalization of the previous result, we state an immediate corollary (our replacement of the estimate \eqref{eq:compare}) which we will use in the next subsection.

\begin{cor}\label{cor:poly_estimate}
Let $f_1, f_2, g$ and $h$ be as in the statement of Proposition \ref{prop:H-1aim}.
Then for every $\mu,\mu_2>0$ and any $\gamma \in (0,1) $ there hold
\begin{equation}\label{eq:compare_modified_1}
\|g(\chi_{1,\mu,\mu_2}(D)f_1)- \chi_{2,\mu,\mu_2}(D)g(f_1)\|^2_{L^2}
\leq \frac{C_0 }{\gamma^{2d}}\max\big\{\big(\mu^{-2}\delta + \mu_2^{-1} \beta\big)^{1-\gamma},\mu^{-2}\delta + \mu_2^{-1} \beta\big\}
\end{equation}
and
\begin{equation}\label{eq:compare_modified_2}
\|h(\chi_{2,\mu,\mu_2}(D)f_2)-\chi_{1,\mu,\mu_2}(D)h(f_2)\|^2_{L^2} \leq \frac{C_0 }{\gamma^{2d}}\max\big\{\big(\mu^{-2}\delta + \mu_2^{-1}\beta\big)^{1-\gamma},\mu^{-2}\delta + \mu_2^{-1}\beta\big\},
\end{equation}
with $C_0>0$ being a constant depending on $\|f_1\|_{L^\infty}$, $\|f_2\|_{L^\infty}$, $g$, $h$ and $d$.
\end{cor}

\begin{proof}
Thanks to the triangle inequality, \eqref{eq:elastic_energy1} and \eqref{eq:Lip1} we have
\begin{align*}
&\|g(\chi_{1,\mu,\mu_2}(D)f_1)-\chi_{2,\mu,\mu_2}(D)g(f_1)\|_{L^2}^2 \\
& \quad \le 2 \|g(\chi_{1,\mu,\mu_2}(D)f_1)-g(f_1)\|_{L^2}^2 + 2 \|g(f_1)-\chi_{2,\mu,\mu_2}(D)g(f_1)\|_{L^2}^2 \\
& \quad \le \frac{2C'^2}{\gamma^{2d}}\big(\|f_1-\chi_{1,\mu,\mu_2}(D)f_1\|_{L^2}^2\big)^{1-\gamma} + 2 C (\mu^{-2}\delta+\mu_2^{-1}\beta) \\
& \quad \le \frac{2C'^2}{\gamma^{2d}}(\mu^{-2}\delta+\mu_2^{-1}\beta)^{1-\gamma} + 2 C (\mu^{-2}\delta+\mu_2^{-1}\beta)
\end{align*}
and therefore \eqref{eq:compare_modified_1}.
An analogous argument leads to \eqref{eq:compare_modified_2}
\end{proof}

We stress that the constant $C_0$ introduced in Corollary \ref{cor:poly_estimate} is the same as that of Proposition \ref{prop:H-1aim} and it is chosen to be greater than $2C+2C'^2+2$, where $C$ and $C'$ are the constants of Lemmas \ref{lem:imp_comp1} and \ref{lem:poly_estimate}, respectively.

\subsection{A bootstrap argument}
\label{sec:lower_bootstrap}

In this section, we carry out our main bootstrap argument. Let us explain the strategy of this before formulating the precise results. It consists of three main steps:

\medskip

\emph{Step 1: The starting point.}
As our starting point, we note that Lemma \ref{lem:imp_comp1} contains the information that the $L^2$-mass of the states $f_1$ and $f_2$ concentrate (in the frequency space) on the compact cones $C_{1,\mu,\mu_2}$ and $C_{2,\mu,\mu_2}$, respectively (Figure \ref{fig:cones1}). It allows us to control the mass of $f_1$, $f_2$ \emph{outside} of these cones.
This information is a direct consequence of the inequalities in \eqref{eq:H-1aim1}, which correspond to elastic energy and surface energy controls. 

\begin{figure}[t]
\includegraphics{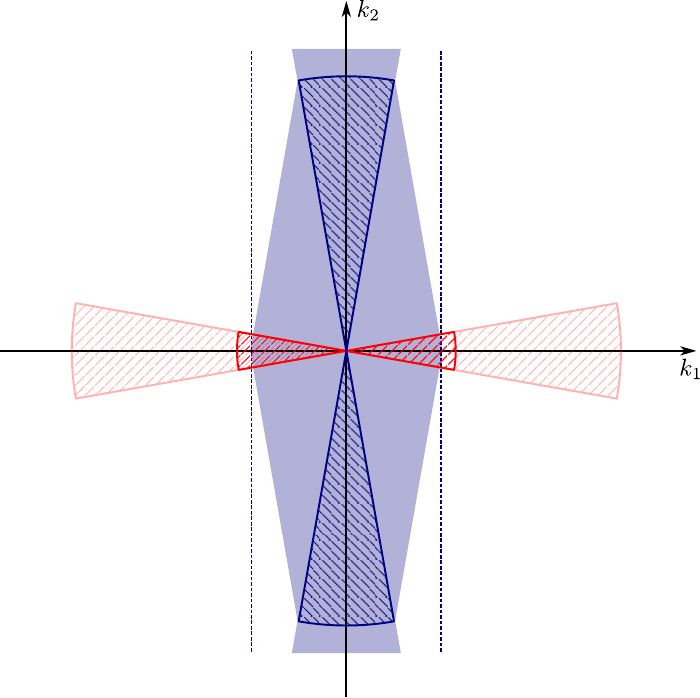}
\quad
\includegraphics{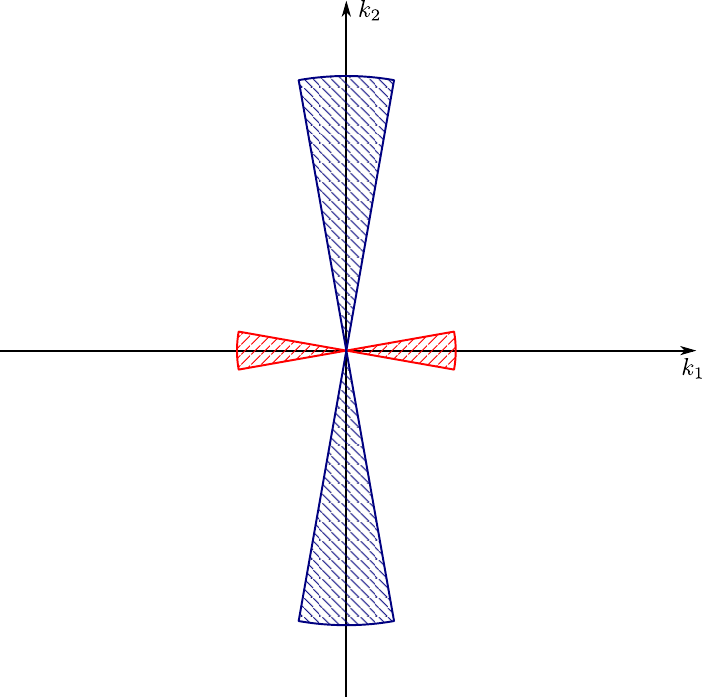}
\caption{The red and blue hashed regions depict $C_{2,\mu,\mu_3}$ and $C_{1,\mu,\mu_2}$, respectively.
The shaded light-blue region represents the Minkowski sum of $C_{1,\mu,\mu_2}$ with itself (obtained as a bound on the mass of the convolution). This implies that the mass of $C_{2,\mu,\mu_2}$ actually concentrates in the smaller red cones $C_{2,\mu,\mu_3}$ instead of the original cone $C_{2,\mu,\mu_2}$.}
\label{fig:cones2}
\end{figure}

\medskip

\emph{Step 2: Exploiting the ``determinedness'' of $f_2$ in terms of $f_1$ in the form of the estimate \eqref{eq:compare}.}
As a next step, we seek to improve the bounds on the mass concentration of $f_1$ and $f_2$ and to iteratively also control the mass of $f_1$ and $f_2$ \emph{inside} of the cones except for possible concentrations at the origin: To this end, we use that the estimates \eqref{eq:elastic_energy1} and \eqref{eq:elastic_energy2} can be improved by noting that $f_2=g(f_1)$ and $f_1=h(f_2)$ with $g$ and $h$ two polynomials. Here, an estimate of the type \eqref{eq:compare} is crucial, since it allows to compare the Fourier supports of $f_1$ and $f_2 = g(f_1)$ by viewing \eqref{eq:compare} as
\begin{align*}
\|g(\chi_{1,\mu,\mu_2}(D)f_1)-\chi_{2,\mu,\mu_2}(D)f_2 \|_{L^2}^2
\lesssim \mu^{-1}\delta + \mu_2^{-1}\beta.
\end{align*}
In particular, this implies that the Fourier support of $f_2$ in the cone $C_{2,\mu,\mu_2}$ is determined by the interaction of the nonlinearity $g$ and the Fourier support of $ f_1$ in the cone $C_{1,\mu, \mu_2}$.
More precisely, heuristically interpreting \eqref{eq:compare} as a proxy for the identity $ \chi_{2,\mu,\mu_2}(k) \F f_2(k)=\F\big(g(\chi_{1,\mu,\mu_2}(D)( f_1))(k)$, we obtain that the $L^2$-mass of $\hat f_2$ is negligible outside a suitable fattening of $C_{1,\mu,\mu_2}$ thanks to the properties of the Fourier transform and  convolution, and our choice of the parameters in the definition of our cones in \eqref{eq:parameters} below, see Figure \ref{fig:cones2}. Indeed, due to the fact that $g$ is a polynomial, the Fourier support of $f_2$ is determined by the Fourier support of $f_1$ through (a multiple) convolution. Its size can thus be estimated by the (multiple) Minkowski sum of the Fourier support of $f_1$ with itself. Now, if the opening angle of the cones is sufficiently small (which is controlled by the parameters $\mu, \mu_2$ in \eqref{eq:parameters} below), the Fourier support of $f_2$ must have been smaller than originally estimated. In other words, the support of $\F f_2$ must be localized in a new, smaller cone $C_{2,\mu,\mu_3}$ and the Fourier mass of $\F f_2$ in $C_{2,\mu,\mu_2}\setminus C_{2,\mu, \mu_3}$ is controlled in terms of the elastic and surface energies.
This observation is made precise and quantified by Lemma \ref{lem:iteration1} below. Technically this step involves slight losses in the estimates due to the fact that our nonlinearities are not globally Lipschitz continuous and arguments as in Lemma \ref{lem:poly_estimate} are required.
\medskip

\begin{figure}[t]
\includegraphics{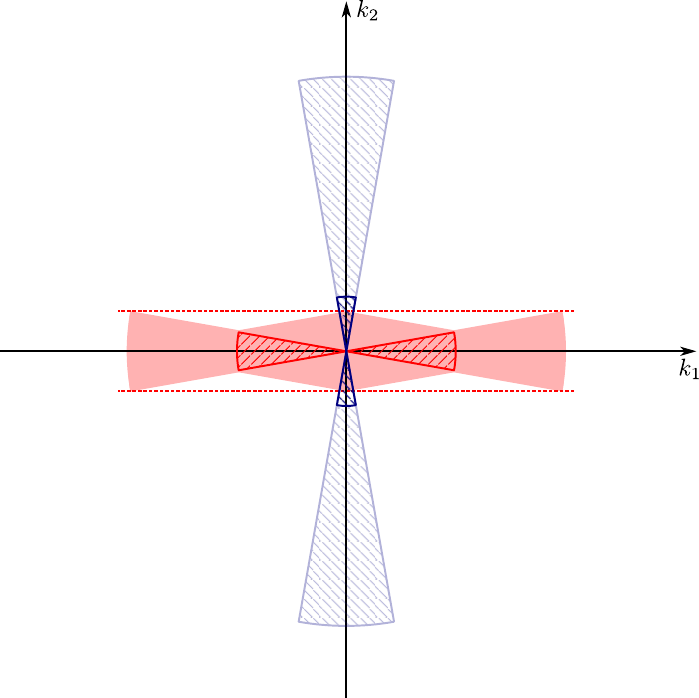}
\quad
\includegraphics{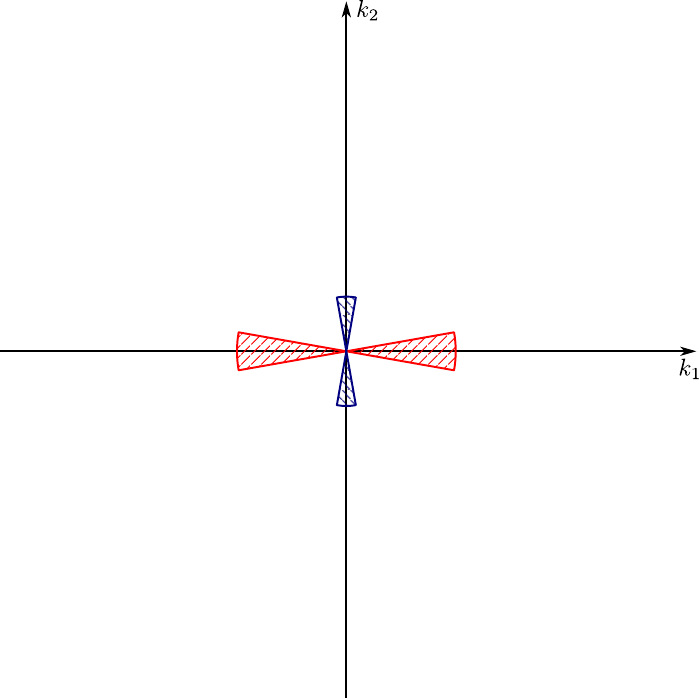}
\caption{The red and blue hashed regions depict $C_{2,\mu,\mu_3}$ and $C_{1,\mu,\mu_4}$, respectively.
The shaded light-red region represents the Minkowski sum of $C_{2,\mu,\mu_3}$ with itself.}
\label{fig:cones3}
\end{figure}

\medskip

\emph{Step 3: Iteration.}
Due to the symmetry of the properties of $f_1$ and $f_2$ it is then possible to obtain a new estimate of the type \eqref{eq:compare}, now with reversed roles for $f_1$ and $f_2$ and for $f_2$ localized to the smaller cone $C_{2,\mu, \mu_3}$
\begin{align*}
\|h(\chi_{2,\mu,\mu_3}(D)f_2)-\chi_{1,\mu,\mu_2}(D)f_1 \|_{L^2}^2
\lesssim \mu^{-1}\delta + \mu_2^{-1}\beta.
\end{align*}
Repeating the Fourier support argument from above with reversed roles for $f_1$ and $f_2$, then also implies that the mass of $f_1$ must concentrate on a smaller cone $C_{1,\mu,\mu_4}$ with $0<\mu_4<\mu_3$ (Figure \ref{fig:cones3}). 

Finally, iterating this process, we obtain that the states $f_1$ and $f_2$ concentrate in smaller and smaller cones in frequency space with corresponding $L^2$-errors which are controlled by elastic and surface energies, see the induction argument in Lemma \ref{lem:induction}.

\medskip

In the following, we make this heuristic argument precise.
To this end, from now on, we fix 
\begin{align}
\label{eq:parameters}
\mu=\epsilon^\alpha \mbox{ and }  \mu_2=\epsilon^{-1+2\alpha}.
\end{align}
Such a choice of the parameters will be clear at the final stage of the argument and will allow us to rewrite the right-hand-side of \eqref{eq:H-1aim2} with a multiple of the total energy.

\begin{lem}\label{lem:iteration1}
Let $f_1$, $g$ and $C_0>0$ be as in the statement of Proposition \ref{prop:H-1aim}.
Then there holds
\begin{multline}\label{eq:iteration1}
\|f_1-\chi_{1,\mu,\mu_2}(D)f_1\|^2_{L^2}+\|g(f_1)-\chi_{2,\mu,\mu_3}(D)g(f_1)\|^2_{L^2}
\\
\leq 4\frac{C_0 }{\gamma^{2d}}\max\big\{\big(\mu^{-2}\delta + \mu_2^{-1} \beta\big)^{1-\gamma},\mu^{-2}\delta + \mu_2^{-1} \beta\big\}.
\end{multline}
where $\mu_3:=\sqrt{2}d\epsilon^{-1+3\alpha}$.
\end{lem}
\begin{proof}
By the choice of the parameters $\mu$ and $\mu_2$ we get
\begin{equation}\label{eq:cones_bounds}
\max_{k\in C_{1,\mu,\mu_2}}|k_1|=\mu_2\mu=\epsilon^{-1+3\alpha}.
\end{equation}
By the properties of Fourier transform and convolution, from the fact that $g$ is polynomial of degree $d$ and from \eqref{eq:cones_bounds} we have that
\begin{equation}\label{eq:supp_bounds}
\begin{split}
\F\big(g(\chi_{1,\mu,\mu_2}(D)f_1)\big)(k)=0 &\quad \text{for } |k_1| > d\epsilon^{-1+3\alpha}.
\end{split}
\end{equation}
Now we define $\chi_{1,\epsilon}$ to be the characteristic function of $\{k\in\R^2 \,:\, |k_1|>d\epsilon^{-1+3\alpha}\}$.
From \eqref{eq:supp_bounds} and \eqref{eq:compare_modified_1} 
we infer that 
\begin{align*}
\|\chi_{1,\epsilon}(D)\chi_{2,\mu,\mu_2}(D)g(f_1)\|^2_{L^2} &= \|\chi_{1,\epsilon}(D)\big(\chi_{2,\mu,\mu_2}(D)g(f_1)-g(\chi_{1,\mu,\mu_2}(D)f_1)\big)\|^2_{L^2} \\
&\le \|\chi_{2,\mu,\mu_2}(D)g(f_1)-g(\chi_{1,\mu,\mu_2}(D)f_1)\|^2_{L^2} \\
&\leq \frac{C_0 }{\gamma^{2d}} \max\big\{\big(\delta \mu^{-2} + \mu_2^{-1}\beta\big)^{1-\gamma},\delta \mu^{-2} + \mu_2^{-1}\beta\big\}.
\end{align*}
This, together with the fact that $|\chi_{2,\mu,\mu_2}-\chi_{2,\mu,\mu_3}|\le\chi_{1,\epsilon}\chi_{2,\mu,\mu_2}$, yields
\begin{equation}\label{eq1:lem_iteration1}
\begin{split}
\|\chi_{2,\mu,\mu_2}(D)g(f_1)-\chi_{2,\mu,\mu_3}(D)g(f_1)\|^2_{L^2} &\le \|\chi_{1,\epsilon}(D)\chi_{2,\mu,\mu_2}(D)g(f_1)\|^2_{L^2} \\
&\le \frac{C_0 }{\gamma^{2d}} \max\big\{\big(\delta \mu^{-2} + \mu_2^{-1}\beta\big)^{1-\gamma},\delta \mu^{-2} + \mu_2^{-1}\beta\big\}.
\end{split}
\end{equation}
Thus, the triangle inequality, \eqref{eq:elastic_energy1} 
and \eqref{eq1:lem_iteration1} give the result:
\begin{align*}
&\|f_1-\chi_{1,\mu,\mu_2}(D)f_1\|^2_{L^2} + \|g(f_1)-\chi_{2,\mu,\mu_3}(D)g(f_1)\|_{L^2}^2\\
&\leq 2 \|g(f_1)-\chi_{2,\mu,\mu_2}(D)g(f_1)\|_{L^2}^2 + 2\|\chi_{2,\mu,\mu_2}(D)g(f_1)-\chi_{2,\mu,\mu_3}(D)g(f_1)\|^2_{L^2} \\
& \quad + \|f_1-\chi_{1,\mu,\mu_2}(D)f_1\|^2_{L^2}\\
&\leq 2C(\mu^{-2}\delta + \mu_2^{-1}\beta) + 2\frac{C_0 }{\gamma^{2d}} \max\big\{\big(\delta \mu^{-2} + \mu_2^{-1}\beta\big)^{1-\gamma},\delta \mu^{-2} + \mu_2^{-1}\beta\big\}&
\\
&\leq 4\frac{C_0 }{\gamma^{2d}} \max\big\{\big(\delta \mu^{-2} + \mu_2^{-1}\beta\big)^{1-\gamma},\delta \mu^{-2} + \mu_2^{-1}\beta\big\}
\end{align*}
using that $C_0\ge C$.
\end{proof}

Next, we iterate this and thus obtain that $f_1$ and $f_2$ can always be approximated by functions with smaller and smaller support in Fourier space (see Figure \ref{fig:cones3}).

\begin{lem}\label{lem:induction}
Let $f_1,f_2,g, h$ and $C_0>0$ be as in the statement of Proposition \ref{prop:H-1aim} and let
\begin{equation}
\mu_m:=(\sqrt{2}d)^m \epsilon^{-1+m\alpha}.
\end{equation}
Then, for every $m\in\N$ there holds
\begin{multline}
\|f_1-\chi_{1,\mu,\mu_{m_e}}(D)f_1\|_{L^2}^2+\|f_2-\chi_{2,\mu,\mu_{m_o}}(D)f_2\|_{L^2}^2 \\
\le \Big(\frac{4 C_0}{\gamma^{2d}}\Big)^{m} \max\big\{\big(\mu^{-2}\delta+\mu_2^{-1}\beta\big)^{(1-\gamma)^m},\mu^{-2}\delta+\mu_2^{-1}\beta\big\},
\end{multline}
where $m_e=2\big\lfloor\frac{m+2}{2}\big\rfloor$ and $m_o=2\big\lfloor\frac{m+1}{2}\big\rfloor+1$ are respectively the lower even and odd parts of $m+2$.
\end{lem}
\begin{proof}
We reason by induction.
The induction basis is provided by Lemma \ref{lem:iteration1}.

Without loss of generality we may assume $m\in2\N$, thus $m_e=m+2$, $m_o=m+1$ and also $(m-1)_e=m$, $(m-1)_o=m+1$.
Assume the inductive hypothesis
\begin{multline}\label{eq:lem_induction1}
\|f_1-\chi_{1,\mu,\mu_m}(D)f_1\|_{L^2}^2+\|f_2-\chi_{2,\mu,\mu_{m+1}}(D)f_2\|_{L^2}^2 \\
\le \Big(\frac{4 C_0}{\gamma^{2d}}\Big)^{m-1}\max\big\{\big(\mu^{-2}\delta+\mu_2^{-1}\beta\big)^{(1-\gamma)^{m-1}},\mu^{-2}\delta+\mu_2^{-1}\beta\big\}
\end{multline}
to hold true. We now show that the statement remains valid for $(m-1) \rightsquigarrow m $.

\emph{Step 1.}
Here, by the triangle inequality, the fact that $f_1 =h (f_2)$ and \eqref{eq:lem_induction1} we get
\begin{equation}\label{eq:lem_induction2}
\begin{split}
&\|h(f_2)-\chi_{1,\mu,\mu_{m+2}}(D)h(f_2)\|_{L^2}^2+\|f_2-\chi_{2,\mu,\mu_{m+1}}(D)f_2\|_{L^2}^2 \\
&\quad \le 2\|h(f_2)-\chi_{1,\mu,\mu_m}(D)h(f_2)\|_{L^2}^2+\|f_2-\chi_{2,\mu,\mu_{m+1}}(D)f_2\|_{L^2}^2 \\
&\quad\qquad +2\|\chi_{1,\mu,\mu_m}(D)h(f_2)-\chi_{1,\mu,\mu_{m+2}}(D)h(f_2)\|_{L^2}^2 \\
&\quad \le 2\Big(\frac{4 C_0}{\gamma^{2d}}\Big)^{m-1}\max\big\{\big(\delta \mu^{-2} + \mu_2^{-1}\beta\big)^{(1-\gamma)^{m-1}},\delta \mu^{-2} + \mu_2^{-1}\beta\big\} \\
&\quad\qquad +2\|\chi_{1,\mu,\mu_m}(D)h(f_2)-\chi_{1,\mu,\mu_{m+2}}(D)h(f_2)\|_{L^2}^2.
\end{split}
\end{equation}

\emph{Step 2.}
We now reason as similarly as in the proof of Lemma \ref{lem:iteration1}.
From
\begin{align*}
\max_{k\in C_{2,\mu,\mu_{m+1}}}|k_2|=\mu_{m+1}\mu
\end{align*}
we infer
\begin{align}
\label{eq:geo}
\F\big(h(\chi_{2,\mu,\mu_{m+1}}(D)f_2)\big)=0 \quad \text{for } |k_2|>d\mu_{m+1}\mu.
\end{align}
Let $\chi_{2,\epsilon}$ denote the characteristic function of $\{k\in\R^2 : |k_2|>\mu_{m+1}\mu\}$.
Thus, from the fact that $|\chi_{1,\mu,\mu_m}-\chi_{1,\mu,\mu_{m+2}}|\le\chi_{2,\epsilon}\chi_{1,\mu,\mu_m}$ and recalling \eqref{eq:geo}, we obtain
\begin{align*}
&\|\chi_{1,\mu,\mu_m}(D)h(f_2)-\chi_{1,\mu,\mu_{m+2}}(D)h(f_2)\|_{L^2}^2 \\
& \quad \le \|\chi_{2,\epsilon}(D)\chi_{1,\mu,\mu_m}(D)h(f_2)\|_{L^2}^2 \\
& \quad \le \|\chi_{2,\epsilon}(D)\big(\chi_{1,\mu,\mu_m}(D)h(f_2)-h(\chi_{2,\mu,\mu_{m+1}}(D)f_2)\big)\|_{L^2}^2 \\
& \quad \le \|\chi_{1,\mu,\mu_m}(D)h(f_2)-h(\chi_{2,\mu,\mu_{m+1}}(D)f_2)\|_{L^2}^2.
\end{align*}
Thus, by the triangle inequality
\begin{multline*}
\|\chi_{1,\mu,\mu_m}(D)h(f_2)-\chi_{1,\mu,\mu_{m+2}}(D)h(f_2)\|_{L^2}^2 \\
\le 2 \|\chi_{1,\mu,\mu_m}(D)h(f_2)-h(f_2)\|_{L^2}^2+2\|h(f_2)-h(\chi_{2,\mu,\mu_{m+1}}(D)f_2)\|_{L^2}^2.
\end{multline*}
We control the first term on the right-hand-side above by means of the inductive hypothesis \eqref{eq:lem_induction1} and the second by \eqref{eq:Lip2} 
and again \eqref{eq:lem_induction1}, that is
\begin{align*}
& \|h(f_2)-h(\chi_{2,\mu,\mu_{m+1}}(D)f_2)\|_{L^2}^2 \le \frac{C'^2}{\gamma^{2d}}\big(\|f_2-\chi_{2,\mu,\mu_{m+1}}(D)f_2\|_{L^2}^2\big)^{1-\gamma} \\
& \quad \le \frac{C'^2}{\gamma^{2d}}\Big(\frac{4C_0}{\gamma^{2d}}\Big)^{m-1}\max\big\{\big(\delta \mu^{-2} + \mu_2^{-1}\beta\big)^{(1-\gamma)^m},\big(\delta \mu^{-2} + \mu_2^{-1}\beta\big)^{1-\gamma}\big\}.
\end{align*}
Hence, recalling that $C_0\ge2+2C'^2$, we infer
\begin{align*}
& \|\chi_{1,\mu,\mu_m}(D)h(f_2)-\chi_{1,\mu,\mu_{m+2}}(D)h(f_2)\|_{L^2}^2 \\
& \quad \le 2 \Big(\frac{4C_0}{\gamma^{2d}}\Big)^{m-1}\max\big\{\big(\delta \mu^{-2} + \mu_2^{-1}\beta\big)^{(1-\gamma)^{m-1}},\delta \mu^{-2} + \mu_2^{-1}\beta\big\} \\
& \qquad + 2\frac{C'^2}{\gamma^{2d}}\Big(\frac{4C_0}{\gamma^{2d}}\Big)^{m-1}\max\big\{\big(\delta \mu^{-2} + \mu_2^{-1}\beta\big)^{(1-\gamma)^m},\big(\delta \mu^{-2} + \mu_2^{-1}\beta\big)^{1-\gamma}\big\} \\
& \quad \le \frac{1}{4}\Big(\frac{4C_0}{\gamma^{2d}}\Big)^m\max\big\{\big(\delta \mu^{-2} + \mu_2^{-1}\beta\big)^{(1-\gamma)^m},\big(\delta \mu^{-2} + \mu_2^{-1}\beta\big)^{1-\gamma}\big\},
\end{align*}
which combined with \eqref{eq:lem_induction2} gives the result.
\end{proof}

\subsection{Proof of Proposition \ref{prop:H-1aim}}
\label{sec:chain_proof}
In this section, we conclude the proof of Proposition \ref{prop:H-1aim} by combining all the bounds from Sections \ref{sec:lower1}-\ref{sec:lower_bootstrap}.

\begin{proof}[Proof of Proposition \ref{prop:H-1aim}]
Consider $m\in2\N$.
From Lemma \ref{lem:induction} we deduce
\begin{equation}\label{eq:proof1}
\|f_1-\chi_{1,\mu,\mu_{m+2}}(D)f_1\|_{L^2}^2 \le \Big(\frac{4 C_0}{\gamma^{2d}}\Big)^m\max\big\{\big(\mu^{-2}\delta+\mu_2^{-1}\beta\big)^{(1-\gamma)^m},\mu^{-2}\delta+\mu_2^{-1}\beta\big\}.
\end{equation}
We first identify the number of iterations $m$ such that $\mu_{m+2}<1$, so that the left-hand-side of \eqref{eq:proof1} reduces to
$$
\sum_{k\neq(0,0)}|\hat f_1(k)|^2=\|f_1-c_0\|_{L^2}^2
$$
with $c_0$ the mean of $f_1$.
The condition $\mu_{m+2}<1$ corresponds to
$$
(m+2) \log(\sqrt{2}d)+(-1+(m+2)\alpha)\log(\epsilon)<0.
$$
This yields $m+2>\frac{1}{\alpha}$ which is satisfied \emph{e.g.} by $m=2\big\lfloor\frac{1}{2\alpha}\big\rfloor$.
For such a choice of $m$ we arrive at
$$
\|f_1-c_0\|_{L^2}^2 \le \Big(\frac{4 C_0}{\gamma^{2d}}\Big)^\frac{1}{\alpha} \max\Big\{\big(\mu^{-2}\delta+\mu_2^{-1}\beta\big)^{(1-\gamma)^\frac{1}{\alpha}},\mu^{-2}\delta+\mu_2^{-1}\beta\Big\}.
$$
We now take $\gamma=\alpha^2$.
Since $\alpha$ is a small parameter (to be determined) $\frac{1}{2}<(1-\alpha^2)^\frac{1}{\alpha}<1$.
Thus, recalling the definition of $\mu:= \epsilon^{\alpha}$ and $\mu_2 = \epsilon^{-1+2\alpha}$, we get
\begin{equation}\label{eq:proof2}
\|f_1-c_0\|_{L^2}^2 \le \Big(\frac{4C_0}{\alpha^{4d}}\Big)^\frac{1}{\alpha} \epsilon^{-2\alpha}\max\big\{\big(\delta+\epsilon\beta\big)^\frac{1}{2},\delta+\epsilon\beta\big\}.
\end{equation}
\end{proof}

\begin{rmk}
We emphasize that there is no particular reason to choose the exponent $\frac{1}{2}$ in the exponent of the right hand side of \eqref{eq:proof2}. It would have been possible to produce any power in $(0,1)$. As this does not play a major role in our estimates below, we have simply chosen this power for convenience.
\end{rmk}

We can further improve the right-hand-side of \eqref{eq:proof2} by noticing that for some constant $c>0$ depending on $d$
$$
\big(\alpha^{-4d}\big)^\frac{1}{\alpha}\lesssim\exp\big(c\log(\alpha^{-1})\alpha^{-1}\big)\le e^{c\alpha^{-1-\nu}}
$$
for every $\nu>0$, which gives
$$
\|f_1-c_0\|_{L^2}^2\lesssim (4C_0e^c)^{\alpha^{-1-\nu}}\epsilon^{-2\alpha}\max\{(\delta+\epsilon\beta)^\frac{1}{2},\delta+\epsilon\beta\}.
$$
Optimizing in $\alpha$, we deduce $(4C_0e^c)^{\alpha^{-1-\nu}}\sim\epsilon^{-2\alpha}$, that is
$$
\alpha\sim|\log(\epsilon)|^{-\frac{1}{2+\nu}}.
$$
We eventually obtain
\begin{equation}\label{eq:proof-improved}
\|f_1-c_0\|_{L^2}^2\lesssim \exp(C|\log(\epsilon)|^{\frac{1}{2}+\nu'})\max\{(\delta+\epsilon\beta)^\frac{1}{2},\delta+\epsilon\beta\},
\end{equation}
for every $\nu'>0$.

\subsection{Application to Tartar's square and proof of the lower bound from Theorem \ref{thm:main}}
\label{sec:lower_proof}

We now consider the case $f_1=\chi_{2,2}$ and $f_2=\chi_{1,1}$, where the phase indicators $\chi_j$ are defined as in \eqref{char-functions}. Using the lower bound from Proposition \ref{prop:H-1aim} we derive the following lower bound for the elastic energy, which, in particular, yields the proof of the lower bound in Theorem \ref{thm:main} for the periodic setting. We refer to the argument below which allows us to then also transfer this to the case of affine boundary conditions.

\begin{thm}
\label{thm:lower_bound}
Let $E_\epsilon$ be as in \eqref{eq:e} and $r_\nu(\epsilon)$ as in \eqref{eq:scalings}. Let $F\in \mathcal{K}^{qc}$. Assume that $E_{\epsilon}(\chi) \leq 1$.
Then, for every $\nu\in(0,1)$ and for every $\chi_j$ for $j=1,\dots,4$ as in \eqref{char-functions} there holds that
$$
r_\nu(\epsilon) \dist^2(F,\mathcal{K}) \lesssim E_\epsilon(\chi)^\frac{1}{2} 
$$
for every $\epsilon>0$ small enough.
\end{thm}

\begin{proof}
From Lemma \ref{lem:e_el-old} and the definition of surface energy \eqref{eq:e_surf}, the inequalities in \eqref{eq:H-1aim1} hold true with
$$
\delta=E_{el}^{\text{per}}(\chi,F) \quad\text{and}\quad \beta=E_{surf}(\chi).
$$
We set $r_\nu(\epsilon):=\exp(-C|\log(\epsilon)|^{\frac{1}{2}+\nu})$.
By Proposition \ref{prop:H-1aim} and the improved estimate \eqref{eq:proof-improved} we infer
$$
\|\chi-\bar{\chi}\|_{L^2}^2 \lesssim r_\nu(\epsilon)^{-1} \max\{ E_\epsilon^{\text{per}}(\chi,F)^\frac{1}{2},E_\epsilon^{\text{per}}(\chi,F)\},
$$
for every $\nu\in(0,1)$, where $\bar{\chi}$ is the mean of $\chi$. 

In order to conclude the argument, we seek to provide a bound on $\dist(\bar{\chi}, F)$. To this end, we invoke the boundary conditions and make use of the elastic energy bounds: For instance,
\begin{align*}
\left|\bar{\chi}_{1,1}-F_{11}\right|^2 
&\leq \left| \int\limits_{[0,1]^2} (\chi_{1,1}(x)- F_{11})dx \right|^2\\
& \leq 2\left| \int\limits_{[0,1]^2} (\p_1 u_1(x)- F_{11})dx  \right|^2 + 2\left| \int\limits_{[0,1]^2} (\chi_{1,1}(x)- \p_{1}u_1 )dx \right|^2
 \leq 4 E_{el}^{\text{per}}(\chi,F),
\end{align*}
where we have used that by the affine boundary conditions $\int\limits_{[0,1]^2} (\p_1 u_1(x)- F_{11})dx =0$. Arguing similarly for the $\chi_{2,2}$ component and invoking the triangle inequality, it follows that for any boundary datum $F \in \mathcal{K}^{qc} \subset \R^{2\times 2}$, we have 
\begin{align*}
\dist^2(F, \mathcal{K}) - 8 E_{el}^{\text{per}}(\chi,F) \lesssim r_{\nu}(\epsilon)^{-1} E_\epsilon^{\text{per}}(\chi,F)^\frac{1}{2},
\end{align*}
Multiplying this inequality with $r_{\nu}(\epsilon)$ and noting that for $\epsilon \in (0,1)$ and $E_{\epsilon}^{\text{per}}(\chi)\leq 1$ there exists $C>0$ such that $r_{\nu}(\epsilon)E_{el}^{\text{per}}(\chi,F) \leq C E_\epsilon^{\text{per}}(\chi,F)^\frac{1}{2} $, this implies the desired claim.
\end{proof}

Theorem \ref{thm:lower_bound} combined with Proposition \ref{prop:upp-bound} proves the main result of this paper, Theorem \ref{thm:main}, in the periodic setting.

Last but not least, we now also transfer the lower bound estimate to the case of affine boundary data:

\begin{proof}[Proof of the lower bound of Theorem \ref{thm:main} in the case of affine boundary conditions]
We first note that $\mathcal{A}^{\text{aff}} \subset \mathcal{A}^{\text{per}}$. Since $(L^{\infty}\cap BV)(\mathbb{T}^2) \subset (L^{\infty}\cap BV)((0,1)^2)$, this implies that for each $\chi \in (L^{\infty}\cap BV)(\mathbb{T}^2)$ it holds that
\begin{align}
\label{eq:min_fin}
E^{\text{per}}_{\epsilon}(\chi) \leq E^{\text{aff}}_{\epsilon}(\chi).
\end{align}
Recalling the trace theorem for $BV$ functions, we further note that any function in $(L^{\infty}\cap BV)((0,1)^2)$ can also be viewed as a function in $(L^{\infty}\cap BV)(\mathbb{T}^2)$ by periodic extension. Hence, \eqref{eq:min_fin} yields that
\begin{align*}
E_{\epsilon}^{\text{per}} \leq E_{\epsilon}^{\text{aff}}.
\end{align*}
Combining this with Theorem \ref{thm:lower_bound} then also concludes the proof of the lower bound estimate in Theorem \ref{thm:main} in the case of affine boundary conditions.
\end{proof}

\section*{Acknowledgements}
Both authors gratefully acknowledge funding by the Deutsche Forschungsgemeinschaft (DFG, German Research Foundation) through SPP 2256, project ID 441068247. 
A.R. is a member of the Heidelberg STRUCTURES Excellence Cluster, which is funded by the Deutsche Forschungsgemeinschaft (DFG, German Research Foundation) under Germany's Excellence Strategy EXC 2181/1 - 390900948.
Both authors would like to thank John Ball for bringing the articles \cite{W97, C99} to our attention.

\bibliography{citations1}
\bibliographystyle{alpha}

\end{document}